\documentclass[leqno,11pt]{article}
\usepackage{microtype}
\usepackage[english]{babel}
\usepackage{amsmath,amsthm,amssymb,mathrsfs} 
\usepackage{ae,aecompl}
\usepackage{times}
\usepackage[pagebackref,hypertexnames=false]{hyperref}
\usepackage{fancybox,fancyhdr,graphics,epsfig}
\usepackage[usenames,dvipsnames]{color}
\usepackage{bbm,subfigure}
\usepackage{verbatim}
\DeclareMathOperator{\dist}{dist}

\DeclareMathOperator{\im}{Im}

\DeclareMathOperator{\Vol}{Vol}

\DeclareMathOperator{\B}{B}

\DeclareMathOperator{\vol}{Vol}

\DeclareMathOperator{\mt}{mt}

\DeclareMathOperator{\conv}{conv}

\DeclareMathOperator{\Var}{Var}

\def\R{\mathbb{R}}

\def\P{\mathbb{P}}

\def\T{\mathbb{T}}
\def\S{\mathbb{S}}

\def\cC{\mathcal{C}}

\def\cP{\mathcal{P}}

\def\cS{\mathcal{S}}

\def\cX{\mathcal{X}}
\def\cY{\mathcal{Y}}

\newcommand{\E}{\mathbb{E}} %expectation
\newcommand{\given}{\;|\;}
\newcommand{\mean}[1] {\E\left\{{#1}\right\}}
\newcommand{\meanx}[1] {\E\{{#1}\}}
\newcommand{\cmean}[2] {\E\left\{#1\given #2\right\}}
\newcommand{\cmeanx}[2] {\E\{#1\given #2\}}
\newcommand{\ind}{\boldsymbol{\mathbbm{1}}} % indicator function

\newcommand{\var}[1]{\mathrm{Var}\param{{#1}}}

\newcommand{\set}[1]{\left\{#1\right\}}

\newcommand{\param}[1]{\left(#1\right)}
\newcommand{\abs}[1] {\left| {#1}\right|}

\newcommand{\prob}[1]{\mathbb{P}\left(#1\right)}

\newcommand{\cprob}[2]{\mathbb{P}\left(#1\given #2\right)} %conditional probability

\newcommand{\eps}{\epsilon}

\newcommand{\by}{\mathbf{y}}

\newcommand{\bv}{\mathbf{v}}
\newcommand{\bw}{\mathbf{w}}

\providecommand{\setthms}[1]{#1}
\setthms{
\newtheorem{lem}{Lemma}[section]
\newtheorem{thm}[lem]{Theorem}
\newtheorem{prop}[lem]{Proposition}
\newtheorem{cor}[lem]{Corollary}

\newtheorem{con}[lem]{Conjecture}
\newtheorem{rem}[lem]{Remark}

\theoremstyle{definition}
\newtheorem{defn}[lem]{Definition}

}
\newcommand{\cech}{\v{C}ech }
\newcommand{\erdos}{Erd\H{o}s }
\newcommand{\renyi}{R\'enyi }

\newcommand{\iid}{\mathrm{i.i.d.}}

\newcommand{\ninf}{n\to\infty}
\newcommand{\pois}[1]{\mathrm{Poisson}\param{{#1}}}

\newcommand{\limninf}{\lim_{\ninf}}

\numberwithin{equation}{section}

\def\bsplit#1\esplit{\begin{split} #1 \end{split} }
\def\beq#1\eeq{\begin{equation} #1 \end{equation}}

\def\splitb#1\splite{\begin{split} #1 \end{split} }
\def\eqb#1\eqe{\begin{equation} #1 \end{equation}}

\title{\cech Complex in Riemannian Manfiolds}

\DeclareMathOperator{\dvol}{dvol}
\DeclareMathOperator{\dmu}{d\mu}

\author{Omer Bobrowski \\   Technion - Israel Institute of Technology \and  Goncalo Oliveira \\ Duke University}

\title{Random \cech Complexes on  Riemannian Manifolds} 

\date{\today}

\begin{document}
\maketitle

%\layout

%===============================================================================

\begin{abstract}
In this paper we study the homology of a random \cech complex generated by a homogeneous Poisson process in a compact Riemannian manifold $M$.
In particular, we focus on the phase transition for ``homological connectivity'' where the homology of 
 the complex becomes isomorphic to that of $M$.
The results  presented in this paper are an important generalization of \cite{bobrowski_vanishing_2015}, from the flat torus to general compact Riemannian manifolds.
In addition to proving the statements related to homological connectivity, the methods we develop in this paper can be used as a framework for translating results for random geometric graphs and complexes from the Euclidean setting into the more general Riemannian one.
\end{abstract}

%===============================================================================

\section{Introduction}

%===============================================================================

\subsubsection*{Motivation}

%===============================================================================

In this paper we continue the work in \cite{bobrowski_vanishing_2015}, and extend the results on the homological connectivity (or the vanishing of homology) in random \cech complexes from the $d$-dimensional flat torus $\T^d$, to general $d$-dimensional compact Riemannian manifolds.

The study of random simplicial complexes is  originated in the seminal result of \erdos and \renyi  \cite{erdos_random_1959} on the phase transition for connectivity in random graphs $G(n,p)$ (with $n$ vertices, and where edges are included independently and with probability $p$) . In their paper, \erdos and \renyi studied these graphs in the limit when $n\to\infty$ and $p=p(n)\to 0$, and showed that the phase transition for connectivity occurs around $p = \log n/n$,  when the expected degree is approximately $\log n$.\footnote{Note that this is the more familiar formulation of the model, while the  one in \cite{erdos_random_1959} is slightly different, yet equivalent.}
Over the past decade, a body of results was established for higher dimensional generalizations of the $G(n,p)$ graph. In these generalizations, graphs are replaced by simplicial complexes, where in addition to vertices and edges we may include triangles, tetrahedra and higher dimensional simplexes. For example, in the \emph{Linial-Meshulam $k$-complex} \cite{linial_homological_2006,meshulam_homological_2009}, one starts with the full $(k-1)$-skeleton on $n$ vertices, and then attaches the $k$-faces independently and with probability $p$. In a different model, called the \emph{clique complex} \cite{kahle_topology_2009,kahle_sharp_2014}, one starts with a random graph $G(n,p)$ and then adds a $k$-face for every $(k+1)$-clique in the graph.
We will refer to these generalizations as random \emph{combinatorial} complexes (See \cite{kahle_topology_2014} for a survey). It turns out that the \erdos-\renyi threshold for connectivity can be generalized to that of ``homological connectivity'', where the higher homology groups $H_k$ become trivial.

In parallel to the study of  combinatorial complexes, a line of research was established for random \emph{geometric} complexes \cite{bobrowski_distance_2014,bobrowski_topology_2014,kahle_random_2011,yogeshwaran_topology_2015,yogeshwaran_random_2016}. This type of complexes generalizes the model of the \emph{random geometric graph} $G(n,r)$ (introduced in \cite{gilbert_random_1961}), where vertices are placed at random in a metric-measure space, and edges are included based on proximity (see \cite{penrose_random_2003}). The main differences between the geometric and the combinatorial models are twofold. Firstly, in geometric complexes  edges and faces added are no longer independent. Secondly, as we shall see in this paper, in geometric complexes the topology of the underlying metric space plays an important role in the behavior of the complex, whereas in the combinatorial models there is no underlying structure. 
In this paper we focus on the random \cech complex $\cC(n,r)$ generated by taking a random point process of size $n$, and asserting that $k+1$ points span a $k$-simplex, if the balls of radius $r$ around them have a nonempty intersection. 
 
Similarly to the study of combinatorial models, we examine the behavior of geometric complexes in the limit when $n\to\infty$ and $r=r(n)\to 0$. The work in \cite{bobrowski_vanishing_2015} established the first rigorous statement about the phase transition describing homological connectivity in random \cech Complexes generated over the $d$-dimensional flat torus $\T^d$. The main goal of this paper is to extend these results from the flat torus to general $d$-dimensional compact Riemannian manifolds. Note, that as opposed to the combinatorial models, for the \cech complex, we do not expect homology to become trivial in the limit, but rather to become isomorphic to that of the underlying manifold.

In addition to its mathematical value, the study of random geometric complexes is motivated by applications in engineering and statistics. A rising area of research called \emph{Topological Data Analysis} (TDA), or \emph{Applied Topology}, focuses on utilizing topology in  data analysis, machine learning, and network modeling (see \cite{carlsson_topology_2009,zomorodian_topological_2007,wasserman_topological_2016} for an introduction). The main idea is to use topological features (e.g.~homology, Euler characteristic, persistent homology) as a ``signature'' for various types of complex high-dimensional data. 
%There are two main advantages for using topology in data analysis. Firstly, topological descriptors can be used as a lower dimensional representation for high-dimensional data. Secondly, topological features are robust to various deformation and manipulations of the data.
Geometric complexes are used often in TDA to translate data points into a combinatorial-topological space, which in turn can be fed into a software algorithm that  calculates its relevant topological properties.
It is therefore desired to develop a solid statistical theory for geometric complexes (see e.g.~\cite{balakrishnan_minimax_2012,bobrowski_topological_2017,chazal_sampling_2009,niyogi_finding_2008}), and an imperative part of this effort is to develop its probabilistic foundations (see e.g.~\cite{adler_crackle:_2014,bobrowski_maximally_2015,duy_limit_2016,owada_limit_2015}). Most of the results on random geometric complexes and graphs to date have been studied for point processes in a Euclidean space. In applications, however, it is commonly assumed that the data lie on (or near) a manifold. The results and methods in this paper provide an important gateway to analyzing such cases. In particular the threshold for homological connectivity is related to the problem of `topological inference' \cite{niyogi_finding_2008,niyogi_topological_2011}, where the goal is to recover topological properties of a manifold from a finite sample.

%===============================================================================

\subsubsection*{Main Result}

%===============================================================================

The main result of this paper is the generalized version of Theorem 5.4 in  \cite{bobrowski_vanishing_2015}.
In the following we assume that $M$ is a $d$-dimensional compact Riemannian manifold, and $\cP_n$ is a homogeneous Poisson process on $M$ with intensity $n$ (see definition in Section \ref{sec:Poisson}). With no loss of generality, and to shorten notation, we will assume that $\vol(M)=1$ as this will only affect the results by an overall scaling constant. In this case, we define $\Lambda := n \omega_d r^d$, where $\omega_d$ is the volume of a unit ball in the $d$-dimensional Euclidean space $\R^d$. For small $r>0$ this quantity approximates the expected number of points inside a ball of radius $r$, and can be thought of as measure of density for the {process} (when $\vol(M)\ne 1$ this should be $n\omega_dr^d/\vol(M)$). The following result is the main theorem of this paper, which is the Riemannian analog of Theorem 5.4 in \cite{bobrowski_vanishing_2015}, describing the phase transition for homological connectivity in terms of $\Lambda$.

%===============================================================================

\begin{thm}\label{thm:main}
Suppose that as $n\to\infty$, $w(n)\to\infty$. 
Then, for $1 \leq k \leq d-1$
\[
\limninf \prob{H_k(\cC(n,r)) \cong H_k(M)} = \begin{cases} 1 & \Lambda = \log n + k\log\log n + w(n),\\
0 & \Lambda = \log n + (k-2) \log\log n - w(n),\end{cases}
\]
\end{thm}

%===============================================================================

Notice that in \cite{bobrowski_vanishing_2015} we further required that $w(n)\gg\log\log\log n$. The generalized proof we use in this paper allows us to avoid this condition.
While the result in Theorem \ref{thm:main} is not tight, it does provide a good estimate for the exact  threshold. For instance, we can deduce that the exact threshold for $H_k$ occurs at $\Lambda = \log n + \alpha_k\log\log n$, with $\alpha_k \in [k-2,k]$ (see the discussion in Section \ref{sec:disc}).
Notice that there are two homology groups which are not covered by this theorem - $H_0$ and $H_d$. For $H_0$, which describes the connectivity of the underlying random geometric graph, it is known that the threshold occurs around $\Lambda = 2^{-d}\log n$.\footnote{This result is proved only for the torus, but similar techniques as we use in this paper could be used to extend it to the general Riemannian case as well.} As for $H_d$, using the Nerve Lemma \ref{lem:nerve} and coverage arguments \cite{flatto_random_1977}, the threshold can be shown to be $\Lambda = \log n + (d-1)\log\log n$. 

Comparing to the results on combinatorial complexes in \cite{kahle_sharp_2014,linial_homological_2006}, the main difference here is that  the phase transition for connectivity cannot be considered as a special case of the higher dimensional homological connectivity, but rather occurs much earlier. Thus, we observe two main stages - the first one is for connectivity ($H_0$), and the second one for all other homology groups ($H_k,\ k\ge 1$). Within the second stage we observe that the different homologies vanish in an orderly fashion. These observations are discussed in detail in \cite{bobrowski_vanishing_2015}.

%===============================================================================

\subsubsection*{Outline of the proof}

%===============================================================================

The proof of Theorem \ref{thm:main}, has a similar outline to the one in \cite{bobrowski_vanishing_2015}, but with considerable geometric adjustments required for the Riemannian case. In fact, the approach we use here for addressing the general setting turns out to be powerful by allowing us to (a) weaken some of the conditions required in \cite{bobrowski_vanishing_2015}, and (b)
prove many other statements for random \cech complexes in the Riemannian setting.
%carry out further calculations beyond the scope of this paper, which were unaccessibl%e before.

In the remaining of this section we wish to outline the steps required for the proof. The first of these is based on the following Pro
 about the expected Betti numbers of $\cC(n,r)$, denoted by  $\beta_k(r)$. This is a Riemannian analogue of Proposition 5.2 in \cite{bobrowski_vanishing_2015}.

%===============================================================================

\begin{prop}\label{prop:main}
Let $\Lambda\to\infty$ and $r\to 0$, in such a way that $\Lambda r \to 0$. Then, for every $1\le k \le d-1$ there exist constants $a_k,b_k$ (that depend on the metric $g$ as well), such that 
\[
a_k n\Lambda^{k-2} e^{-\Lambda} \le \mean{\beta_k(r)} \le \beta_k(M) + b_k n \Lambda^k e^{-\Lambda}.
\]
%where $\beta_k(r)$ denotes the Betti numbers of $\cC(n,r)$ and $\beta_k(M)$ those of $M$.
\end{prop}

%===============================================================================

Notice that the original version in \cite{bobrowski_vanishing_2015}, stated  for the flat torus, does not require the assumption that $\Lambda r \to 0$. Even though this condition is necessary to extend the statements to the non-flat case, it does not affect the result in any of the radii of interest for us, since for $\Lambda \gg \log n$ the manifold is covered with high probability which implies that $\cC(n,r) \simeq M$.

%The outline of the paper is as follows. Sections 2-5 provide various preliminaries needed for the analysis of the \cech complex in the general Riemannian setting.
The proof of proposition \ref{prop:main} will be split between Section \ref{sec:upper} for the upper bound, and Section \ref{sec:lower} for the lower bound. 
The proof of the upper bound makes use of a special variant of  Morse theory for the distance function, which we will discuss in detail later. The main idea is to use the Morse inequalities to obtain an upper bound on the Betti numbers $\beta_k(r)$ from the number of index $k$ critical points of an appropriate Morse function. To prove the lower bound, we search for special configurations named ``$\Theta$-cycles'' (introduced in \cite{bobrowski_vanishing_2015}), which are guaranteed to generate nontrivial $k$-cycles in homology. Then, by counting these we obtain a lower bound on $\beta_k(r)$.
The proof of Proposition \ref{prop:main} also covers the upper part of the phase transition in Theorem \ref{thm:main}. To prove the lower part, in addition to the first moment bound provided in proposition \ref{prop:main}, we need to use a second moment argument, which will be discussed in Section \ref{sec:second_moment}. In Section \ref{sec:proof} we will put together all the parts of the proof for Theorem \ref{thm:main}.
Before we get to the proofs, we need to provide various statements that will allow us to carry out the calculations in the general Riemannian setting. This will be done in Sections 2-5.

%===============================================================================

\section{Preliminaries}

%===============================================================================

\subsection{Homology}

%===============================================================================

In this section we introduce the concept of homology in an intuitive rather than a rigorous way. A comprehensive introduction to the topic can be found in \cite{hatcher_algebraic_2002,munkres_elements_1984}. The reader familiar with the fundamentals of algebraic topology is welcome to skip to section \ref{ss:TheCechComplex}.

Let $X$ be a topological space, the \textit{homology} of $X$ is a sequence of abelian groups denoted  $\set{H_i(X)}_{i=0}^d$. 
In this paper, we will assume that homology computed with field coefficients, and then the homology groups $H_*(X)$ are in fact vector spaces. This sequence of vector spaces encapsulates topological information about $X$ in the following way. The basis elements of $H_0(X)$ correspond to the connected components of $X$, and for $k\ge 1$ the basis elements of $H_k(X)$ correspond to (nontrivial) $k$-dimensional cycles. Loosely speaking, a nontrivial $k$-dimensional cycle in a manifold $M$ can be thought of as the boundary of a $k+1$-dimensional solid, such that the interior of the solid is not part of  $M$. The Betti numbers are defined as $\beta_k(X) = \dim(H_k(X))$, namely they are the number of linearly independent $k$-dimensional cycles in $X$.

For example, for the $d$-dimensional sphere $\S^d$ we have $\beta_0(\S^d) = \beta_d(\S^d) = 1$, while $\beta_k(\S^d) = 0$ for $k\ne 0,d$. Another example is the $2$-dimensional torus $\T^2$ that has $\beta_0(\T^2) = 1$ (a single component), $\beta_1(\T^d)=2$ (two essential ``loops''), and $\beta_2(\T^2) = 1$ (the boundary of the 3D solid). See Figure \ref{fig:torus}

\begin{figure}[h]
\centering
\includegraphics[scale=0.4]{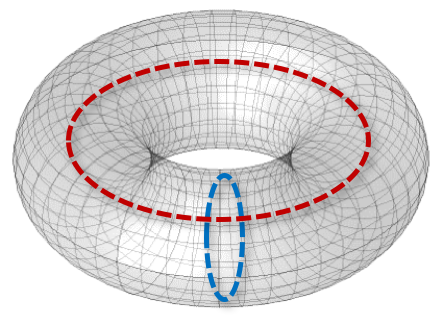}
\caption{\label{fig:torus}  The 2D torus $\T^2$. We observe the two loops generating $H_1$, while the entire 2D surface generates $H_2$.}
\end{figure}

%Homology is a \emph{topological invariant}, namely if $f:X\to Y$ is a homeomorphism, then it induces an isomorphism $f_*:H_*(X)\to H_*(Y)$ between the homology groups.

\subsection{The \cech complex}\label{ss:TheCechComplex}

The \cech complex we study in this paper, is an abstract simplicial complex constructed from a finite set of points in a metric space in the following way.

%%%%%%%%%%

\begin{defn}\label{def:cech_complex}
Let $\cP = \set{p_1,p_2,\ldots,p_n}$ be a collection of points in a metric space, and let $r>0$ and let $B_r(x)$ be the ball of radius $r$ around $x$. The \cech complex $\cC_r(\cP)$ is constructed as follows:
\begin{enumerate}
\item The $0$-simplexes (vertices) are the points in $\cP$.
\item A $k$-simplex $[p_{i_0},\ldots,p_{i_k}]$ is in $\cC_r(\cP)$ if $\bigcap_{j=0}^{k} {B_{r}(p_{i_j})} \ne \emptyset$.

\end{enumerate}
\end{defn}

%%%%%%%%%%

%Figure \ref{fig:cech} presents an example for a \cech complex in $\R^2$.

\begin{figure}[h]
\centering

\includegraphics[scale=0.4]{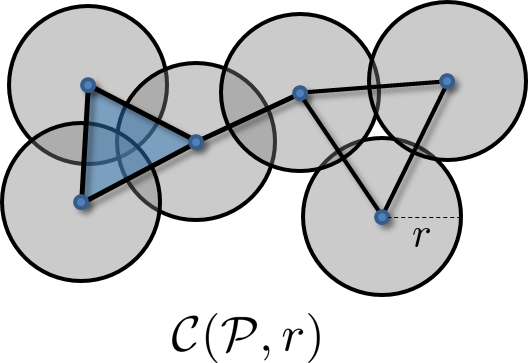}
\caption{\label{fig:cech}  A \cech complex generated by a set of points in $\R^2$. The complex has 6 vertices (0-simplexes), 7 edges (1-simplexes) and a single triangle (a 2-simplex). }
\end{figure}

Associated with the \cech complex $\cC_r(\cP)$ is the union of balls used to generate it (in the underlying metric space), which we define as
\begin{equation}\label{eq:union_balls}
	B_r(\cP) = \bigcup_{p\in \cP}B_r(p).
\end{equation}
The spaces $\cC_r(\cP)$ and $B_r(\cP)$ are of a completely different nature. Nevertheless, the following lemma claims that they are very similar in the topological sense.
This lemma is a special case of a more general topological statement originated in \cite{borsuk_imbedding_1948} and commonly referred to as the `Nerve Lemma'.

%%%%%%%%%%

\begin{lem}[The Nerve Lemma]\label{lem:nerve}
Let $\cC_r(\cP)$ and $B_r(\cP)$ as defined above. If for every $p_{i_1},\ldots,p_{i_k}$ the intersection $B_r(p_{i_1})\cap \cdots\cap B_r(p_{i_k})$ is either empty or contractible (homotopy equivalent to a point), then $\cC_r(\cP)\simeq B_r(\cP)$, and in particular,
\[
	H_k(\cC_r(\cP)) \cong H_k(B_r(\cP)),\quad \forall k\ge 0.
\]
\end{lem}

%%%%%%%%%%

Consequently, for sufficiently small $r$, we will sometimes be using $B_r(\cP)$ to make statements about $\cC_r(\cP)$. This will be very useful especially when coverage arguments are available. Note that in Figure \ref{fig:cech} indeed both $\cC_r(\cP)$ and $B_r(\cP)$ have a single connected component and a single hole.

\subsection{The Poisson process}\label{sec:Poisson}
In order to make a \cech complex $\cC_r(\cP)$ random, we generate the point set $\cP$ at random. In this paper we will use the model of a homogeneous Poisson process. Let $M$ be a compact Riemannian manifold, and let $X_1,X_2,\ldots$, be a sequence of $\iid$ (independent and identically distributed) random variables, distributed uniformly on $M$ with respect to the Riemannian volume measure. Let $N\sim \pois{n}$ be a Poisson random variable, independent of the $X_i$-s, and define the spatial Poisson process $\cP_n$ as
\[
\cP_n = \set{X_1,X_2,\ldots, X_N}.
\]
In other words, we generate a random number $N$ and then pick the first $N$ random points generated on $M$. Notice that we have $\mean{\abs{\cP_n}} = n$, so while the number of points is random its expected value is $n$.
For every subset $A\subset M$ we define $\cP_n(A) = \abs{\cP_n\cap A}$, i.e.~the number of points lying in $A$.
The process $\cP_n$ can be equivalently defined as the process that satisfies the following properties:
\begin{enumerate}
\item For every subset $A\subset M$ the number of points in $A$ has a Poisson distribution. More specifically: $\cP_n(A) \sim \pois{n\vol(A) / \vol(M)}$.

\item If $A,B\subset M$ are two disjoint sets ($A\cap B = \emptyset$) then the variables $\cP_n(A),\cP_n(B)$ are independent.
\end{enumerate}

The last property of the Poisson process is known as ``spatial independence'' and it is the main reason why the Poisson process such a convenient model to analyze.

In this paper we study the random \cech complex $\cC(n,r) := \cC_r(\cP_n)$. To shorten notation, from here on we will use $\cC_r$ to refer to $\cC(n,r)$. As mentioned earlier, throughout the paper we will assume that $\vol(M)=1$, to simplify the calculations.	

%%%%%%%%%%

\section{Riemannian Geometry Ingredients}\label{sec:Preliminaries}

In this section we wish to provide the reader with a brief review of Riemannian geometry, which will be used throughout this paper. There are many good textbooks on the subject, for a comprehensive introduction see for example \cite{doCarmo}, \cite{Peterson}, or \cite{Kobayashi}.

%%%%%%%%%%

\subsection{A quick intro to Riemannian geometry}

A Riemannian manifold is a pair $(M,g)$, where $M$ is a smooth manifold and a smoothly varying metric $g:T_pM\times T_pM\to \R$, where $T_pM$ is the tangent space to $M$ at a point $p$. The metric yields a smoothly varying inner product in the tangent space $T_pM$ at any point $p \in M$. Hence, it can be used to define the norm of a tangent vector, as well as angles between tangent vectors at the same point. Using the notion of norm $\abs{ v } = \sqrt{g_p(v,v)}$ of a vector $v \in T_pM$ one can define the length $\ell(\gamma)$ of a path $\gamma : I \subset \mathbb{R} \to M$ by simply integrating its velocity, i.e. 
$$\ell(\gamma) = \int_I \abs{ \dot{\gamma} (t) } \ dt ,$$
where $\dot{\gamma}(t) \in T_{\gamma(t)} M$ denotes the tangent vector to $\gamma$ at $\gamma(t)$, for $t \in I$. Given two points $p,q \in M$ one can define the distance $\dist( p , q)$ as the infimum of the length over all paths connecting $p$ and $q$. A Riemannian manifold is called \emph{ complete} if for any two points $p,q \in M$ there is a curve connecting them which minimizes the distance, i.e.~the infimum length is achieved. These minimizing  curves are called \emph{geodesics}. Notice that every compact manifold is also complete, and moreover, by the theorem of Hopf and Rinow \cite{HopfRinow} it is also a complete metric space. Hence, the definitions of the Poisson process and the \cech complex described above are valid.

The local invariants of a Riemannian metric $g$ are encoded in its curvature. We shall now introduce this in a way that will be useful for us later. Given a point $p \in M$ one can consider the geodesics that pass through $p$ at time $t=0$ with velocity vector a given vector $v \in T_pM$. The geodesic equations are a system of second order  differential equations and a simple application of Picard's existence and uniqueness theorem shows that these geodesic are unique, and we denote them by $\gamma_v(t)$. We can then consider a map $\exp_p:T_pM \rightarrow M$, called the exponential map at $p$, which assigns to each vector $v \in T_pM$ the position at which the unique geodesic starting at $p$ with velocity $v$ is at time $1$, i.e. $\exp_p(v)=\gamma_v(1)$. The derivative of $\exp_p$ at the origin in $T_pM$ is the identity and so, by the inverse function theorem, $\exp_p$ is a local diffeomorphism of a small ball around $0 \in T_pM$ to a small neighborhood of $p \in M$. As a consequence, one can use the exponential map to define local coordinates around $p$. For instance, fixing an orthonormal basis for $T_pM$ one obtains coordinates on $T_pM$, which  can be regarded as local coordinates $(x^1,...,x^d)$ around $p$. These are the so called \emph{geodesic normal coordinates} and have the following properties. The point $p$ corresponds in this coordinates to the point $(0, \ldots , 0)$. Any geodesic through $p$ corresponds to a straight line passing through the origin $(0,...,0)$ and intersects the spheres $S_r(p)= \lbrace (x^1, ... , x^d) \ \vert  (x^1)^2+...+(x^d)^2=r^2 \rbrace$ orthogonally. In these coordinates, the metric can be written as $g=g_{ij} dx^i \otimes dx^j$, with
\begin{equation}\label{eq:NormalCoords}
g_{ij}= \delta_{ij}+ \frac{1}{3} R_{iklj}x^k x^l + O(\abs{ x }^3),
\end{equation}
where $\delta_{ij}$ is the Kronecker delta.
The second order terms, i.e. $R_{iklj}$, in the Taylor expansion above, form the \emph{Riemann curvature tensor} at $p$. 
We note that here, as well as later in the paper, we use  Einstein's notation for summation, where by $a_i b^i$ we mean the sum $\sum_i a_i b^i$.
We refer the reader to any textbook on Riemannian geometry for all that was described above. In what follows we shall give a self contained exposition of all the geometric input needed. This will be mostly derived in the rest of this section and the appendix.

%%%%%%%%%%

\subsection{Notation}

%%%%%%%%%%

In this paper we restrict to the case when $(M,g)$ is a compact Riemannian manifold. For convenience, we will sometimes use $\langle \cdot , \cdot \rangle_p$ to denote  the inner product of tangent vectors at a point $p \in M$ (instead of $g$). Next, for $\cP\subset M$ and $r>0$, we define the following:
\begin{itemize}
\item $B_r(p) =$ the closed ball of radius $r$ around $p$.
\item $S_r(p) =$ the sphere of radius $r$ around $p$.
\item $B_r(\cP) = \bigcup_{p\in \cP} B_r(p)$ - the union of  balls.
\item $B_r^\cap(\cP) = \bigcap_{p\in \cP} B_r(p)$ - the intersection of balls.
\end{itemize}

%===============================================================================

\subsection{Approximations of Riemannian volumes}\label{sec:ineqs}

%===============================================================================

Throughout the proofs in this paper, we will need to approximate volumes for subsets of Riemannian manifolds, and compare them to their Euclidean counterparts. In this section we introduce the Riemannian normal coordinates and use them to provide the main inequalities that we will use for that purpose. The proofs for the statements in this section appear in Appendix A. For an exposition of the background leading to these see for example \cite{Kobayashi,Peterson}. A  particularly nice exposition of Riemannian normal coordinates can be found in \cite{Viaclovsky}.

Let $p \in M$ and $(x^1,...,x^d)$ be geodesic normal coordinates centered at $p$. Then, in the domain of definition of these coordinates the metric can be written as $g=g_{ij} dx^i \otimes dx^j$, with $g_{ij}$ as in  \eqref{eq:NormalCoords}. Then, there is a canonical measure on $M$ induced by the Riemannian density which is given by
$$\abs{ \dvol_g } = \sqrt{ \abs{ \det(g_{ij}) } }  \abs{ \dvol_{g_E} },$$
where $\abs{ \dvol_{g_E} }= \abs{ dx^1\wedge ... \wedge dx^d }= dx_1 \ldots dx_n$ is the density associated with the Euclidean metric with $g_E=\delta_{ij} dx^i \otimes dx^j$. For more on densities in Riemannian manifolds see, for example, pages 304--306 in \cite{Lang}.
In this section we will prove various inequalities that relate geometric quantities associated with $g$ to those associated with $g_E$. The first  ingredient we need is the computation of the Riemannian density $\sqrt{ \abs{ \det(g_{ij}) }}$ using equation \eqref{eq:NormalCoords},

\begin{equation} \label{eq:RiemannianMeasure}
\sqrt{ \abs{ \det(g_{ij}) }} =   1- \frac{Ric_{ij}}{3} x^i x^j + O(\abs{ x }^3),
\end{equation}
where  $Ric_{ij}=-\sum_k R_{ikkj}$ is known as the \emph{Ricci curvature tensor} at $p$. Using this expression for the Riemannian density it is easy to show that the volume of a normal ball $B_r(p)=\lbrace (x^1, ... , x^d) \ \vert \  (x^1)^2+...+(x^d)^2 \leq r^2 \rbrace$ can be written as:

\begin{equation}\label{eq:ball_vol}
\vol(B_r(p)) = \omega_d r^d \left( 1- \frac{s(p)}{6(d+2)}r^2 + O(r^3) \right),
\end{equation}
where $\omega_d$ is the volume of an Euclidean $d$-dimensional ball of radius $1$, and $s(p)$ denotes the \emph{scalar curvature} $s(p)= \sum_i Ric_{ii}$ at $p$. Similarly, one can compute the volume of a normal sphere to be 

\begin{equation}\label{eq:sphere_vol}
\vol(S_r(p)) = d \omega_d r^{d-1} \left( 1- \frac{s(p)}{6d}r^2 + O(r^3) \right).
\end{equation}

In the following, we will make use of equations \eqref{eq:ball_vol}--\eqref{eq:sphere_vol} in order to get bounds on these volumes for small values of $r$.
Notice that we can write $\abs{ \dvol_g }$ in polar coordinates as $\abs{ \dvol_g } =  dr \abs{ \dvol_{S_r(p)} }$, where $\dvol_{S_r(p)}$ denotes the volume form of the induced metric on the  $S_r(p)$, and by $\abs{ \dvol_{S_r(p)} }$ we simply mean that we regard it as a density. In the Euclidean case, $\dvol_{g_E}=r^{d-1} dr \wedge \dvol_{{\mathbb{S}^{d-1}}}$, where ${\mathbb{S}^{d-1}}$ is  the unit round sphere. The following result compares $\dvol_{{S_r(p)}}$ with $r^{d-1} \dvol_{{\mathbb{S}^{d-1}}}$, for a given Riemannian metric $g$. 
%The proof is given in Appendix \ref{appendix1}.

%===============================================================================

\begin{lem}\label{lem:approx_sphere_vol}
Let $(M,g)$ be a compact Riemannian manifold. Denote by $\abs{ Ric_p } = \sup_{v \in T_pM \backslash 0} \frac{ \abs{ Ric(v,v) }}{\abs{ v }^2}$ the norm of the Ricci tensor at $p \in M$. Then, for any $\nu>0$ there exists $r_{\nu}>0$ such that for all $r \leq r_{\nu}$ and for all $p \in M$, we have
$$r^{d-1} \left( 1 - \frac{{\abs{ Ric_p } }+ \nu}{3}\cdot r^2 \right) \abs{ \dvol_{{\mathbb{S}^{d-1}}} } \leq \abs{ \dvol_{{S_r(p)}}  } \leq r^{d-1} \left( 1 + \frac{\abs{ Ric_p } + \nu}{3}\cdot r^2 \right) \abs{ \dvol_{{\mathbb{S}^{d-1}}} },$$
on $B_{r}(p)$. Moreover, $r_{\nu}$  depends continuously on $\nu$.
\end{lem}

%=============================================================================

%===============================================================================

\begin{cor}\label{cor:Volume}
For $\nu>0$, let $s_{min}(\nu)= \inf_{p \in M} \frac{s(c)}{6(d+2)}- \nu$ and $s_{max}(\nu)= \sup_{p \in M} \frac{s(c)}{6(d+2)}+ \nu$. Then, for all $\nu>0$ there exists $r_{\nu}>0$, such that for all $r \leq r_{\nu}$
$$\omega_d r^d \left( 1 - s_{max}(\nu) r^2 \right) \leq \vol(B_r(p)) \leq  \omega_d r^d \left( 1 - s_{min}(\nu) r^2  \right).$$
\end{cor}

%===============================================================================

\begin{lem}\label{lem:RiemannianMeasureComparison}
Let $(M,g)$ be a compact Riemannian manifold. Then, for all $\nu>0$ there is $r_{\nu}>0$, such that for all $p\in M$ and $r \leq r_{\nu}$, we have
\[
(1- \nu  r^2 ) \abs{ \dvol_{g_E} } \leq \abs{ \dvol_g } \leq (1+ \nu r^2) \abs{ \dvol_{g_E} },
\]
on $B_{r}(p)$, for all $p \in M$.
\end{lem}

%===============================================================================
\noindent Note, that we can take $r_{\nu}$ in both lemmas above to be the same (by taking the minimum of the two).\\
The following lemma is an immediate consequence of the above.
The last comparison we will need is between the union of Riemannian balls and Euclidean ones centered at the same points. The statement requires a little bit of notation. For small $r>0$ and $p_1,p_2$ at distance at most $2r$ from each other, we let $p$ be the midpoint of the minimizing geodesic connecting $p_1$ to $p_2$. Next, fix the Riemannian normal coordinates $(x^1, \ldots , x^d)$ centered at $p$, which induce a Euclidean metric $g_E = \delta_{ij} dx^i \otimes dx^j$. In the domain of definition of $g_E$ we will use $B^E_s(q)$ to denote the $s$-ball centered at $q$, where $s$ is measured using the metric $g_E$. 
%We turn now to the statement of the result, as the in the previous cases the proof can be found in Appendix \ref{appendix1}.

%===============================================================================

\begin{lem}\label{lem:RiemannianBallsComparison}
Let $(M,g)$ be a compact Riemannian manifold. Then, there exist $\nu>0$ and $r_{\nu}>0$ such that for every $r<r_\nu$ and any two points $p_1, p_2$ with $\dist(p_1,p_2)<2r$ we have,
$$\left ( B^E_{(1- \nu r)r}(p_1) \cup B^E_{(1- \nu r)r}(p_2) \right) \subset \left( B_{r}(p_1) \cup B_{r}(p_2) \right) \subset \left( B^E_{(1+ \nu r)r}(p_1) \cup B^E_{C(1+ \nu r)r}(p_2) \right).$$
\end{lem}
Note that the values of $r_{\nu}$ in the previous two lemmas can be chosen independently of each other.

%===============================================================================

\section{Morse theory for the distance function}\label{sec:Morse}

%===============================================================================

In this section  we use the results in \cite{gershkovich_morse_1997} to develop the framework which will allow us to justify various Morse theoretic arguments later. For an introduction to Morse theory see \cite{milnor_morse_1963}. Briefly, developing a Morse theoretic framework will allow us to study the homology of the \cech complex by examining critical points for the corresponding distance function (defined below).

Recall that for $x,y\in M$ we define $\dist(x,y)$ to be the geodesic distance between $x$ and $y$ with respect to the metric $g$.
For $p\in M$ we define $\rho_p:M\to \R^+_0$ to be the function $\rho_p(x) := \dist(p,x)$.
If $\cP\subset M$ is a finite set, we define the distance function $\rho_\cP:M\to \R^+_0$ as:
\[
	\rho_{\cP}(x) := \min_{p\in \cP} \rho_p(x).
\]
We shall now establish Morse theory for the function $\rho_{\cP}$, using the framework of Morse theory for min-type functions developed in \cite{gershkovich_morse_1997}. We start by showing that for any finite set $\cP$ the function $\rho^2_{\cP}$ is a \emph{Morse min-type function}, namely that at every point we can  write $\rho_{\cP}^2$ as a minimum of finitely many smooth Morse functions.

%===============================================================================

\begin{lem}\label{lem:epsilon}
Given a compact Riemannian manifold $(M,g)$ there exists $r_{\mt}>0$ such that for every $\cP\subset M$ the function $\rho^2_{\cP}$ is a Morse min-type function on $B_{r_{\mt}}(\cP)$.
%In particular, if $\mathcal{P} \subset M$ is such that $B_{r_{\max}}(\cP) = M$, then $\rho_{\cP}$ is of Morse min-type on all of $M$.
\end{lem}

%===============================================================================

\begin{proof}
For any $p \in M$,  there exists $r_p>0$ such that the function $\rho^2_p( \cdot)$ is smooth, Morse, and strictly convex on $B_{r_p}(p)$. Since the metric $g$ is smooth we can choose $r_p$ continuously in $p \in M$, and since  $M$ is compact we can define $r_{\textnormal{mt}} := \min_{p\in M} r_p >0$. The result follows.
\end{proof}

%===============================================================================

Now that we established that $\rho_{\cP}^2$ is a Morse min-type function, we want to explore its critical points.
Similarly to the Euclidean case (cf. \cite{bobrowski_distance_2014}),  critical points of index $k$ are generated by subsets $\cY\subset \cP$ containing $k+1$ points, and are located at the ``center'' of the set. While in the Euclidean case the center of $\cY$ is simply taken to be the center of the unique $(k-1)$-sphere that contains $\cY$, in the general Riemannian case we need to carefully define the notion of a center.

%===============================================================================

\subsection{The center and radius of a set.}

%===============================================================================

Let $\cY$ be a finite subset of $M$ and define:
\[\begin{split}
	E(\cY) &:= \lbrace x \in M \ \vert \  \rho_{p_1}(x)= \rho_{p_2}(x) = \cdots = \rho_{p_k}(x)) \rbrace,\\
	E_{r}(\cY) &:= {E(\cY)\cap B_{r}(\cY).}
\end{split}\]
In other words, $E(\cY)$ is the set of all points that are equidistant from $\cY$, and $E_r(\cY)$ is its restriction to a small neighborhood around $\cY$. We say that a subset $\cY$ is \emph{generic} if $E(\cY)\ne \emptyset$, and the zero level sets of the $k-1$ functions $\rho^2_{p_i}(\cdot)-\rho^2_{p_1}(\cdot)$ intersect transversely\footnote{Two submanifolds $N_1$ and $N_2$ of $M$ are said to intersect transverly at a point $p$, if $T_p N_{1} + T_p N_2 = T_pM$. They are said to intersect transversely if they intersect transversely at all points $p \in N_1 \cap N_2$.} in $B_{r_{\mt}}^\cap(\cY)$. This guarantees that for a generic set the intersection $E_{r_{\mt}}(\cY)$ is a smooth submanifold of dimension $d-k+1$.

%===============================================================================

\begin{lem}\label{lem:center}
There exists a positive $r_{\max}< r_{mt}$ such that if $\cY\subset M$ is a finite generic subset with $E_{r_{\max}}(\cY) \ne \emptyset$, 
then there exists a unique point $c(\cY) \in M$ such that  for all $p\in \cY$,
\begin{equation}\label{eq:center}
\rho_{p}(c(\cY))= \inf_{x \in E(\cY)}\rho_{\cY}(x).
\end{equation}
In that case, we define $\rho(\cY) := \rho_{\cY}(c(\cY))$, and we will refer to $c(\cY)$ and $\rho(\cY)$ as the center and radius of the set $\cY$, respectively.
\end{lem}

%===============================================================================

\begin{proof}
Since $\cY\subset M$ is generic, $E_{r_{\mt}}(\cY)$ is a nonempty smooth submanifold of dimension $d-k+1$. Moreover, in $E_{r_{\mt}}(\cY)$ we know that $\rho_{\cY} \le r_{\mt}$, with equality attained only on the boundary. Thus, as $E_{r_{mt}}(\cY)$ is bounded and $\rho(\cY)$ is continuous, there is a minimum $c(\cY)$ of $\rho(\cY)$ in $E_{r_{\mt}}(\cY)$. Moreover, since in $E(\cY)\backslash E_{r_{\mt}}(\cY)$ we have $\rho_{\cY} > r_{\mt}$, then $c(\cY)$ is actually a global minimum of $\rho_{\cY} \vert_{E(\cY)}$. In addition, it follows from the definition of $E(\cY)$ that for any $p \in \cY$ we have $\rho_{\cY}|_{E(\cY)} \equiv \rho_p|_{E(\cY)}$, and so equation \eqref{eq:center} holds.\\
To prove that $c(\cY)$ is unique, we use Lemma \ref{lem:Strictly_Convex} in Appendix \ref{appendix2}. This lemma shows that for $r' \le r_{\mt}$ sufficiently small $E_{r'}(\cY)$ approaches a totally geodesic submanifold, as the Riemannian distance functions approach the Euclidean ones, and therefore the restriction of $\rho_{p}$ to $E_{r'}(\cY)$ is strictly convex. As a consequence, $\rho_{p}$ has a unique minimum in $E_{r'}(\cY)$. Since $\rho_{\cY}|_{E(\cY)} \equiv \rho_p|_{E(\cY)}$ (for any $p\in \cY$), and since $\rho_\cY > r'$ on $E(\cY)\backslash E_{r'}(\cY)$, we conclude that $c(\cY)$ is the unique minimum of $\rho_{\cY}$ in $E(\cY)$. Since $M$ is compact so is $M^k$, for $k\in \mathbb{N}$, hence we can minimize the value of $r'$ over all $k$-tuples $\cY \subset M$ as in the statement. From here on we will define $r_{\max}$ to be this minimum value of $r'$. 
\end{proof}

%===============================================================================

\begin{rem}
Given $r_{\max}$ as in Lemma \ref{lem:center} we have that:
\begin{itemize}
\item For any $\cP$, $\rho_{\cP}^2$ is a Morse min-type function on $B_{r_{\max}}(\cP)$. 
\item For every $\cY\subset \cP$ with $E_{r_{\max}}(\cY)$ nonempty, $c(\cY)$ is uniquely defined. 
\item For $r \leq r_{max}$, the exponential map $\exp_p$ is always defined for $v \in T_pM$ with $\abs{ v } \leq r$.
\end{itemize}
\end{rem}

%===============================================================================

\begin{rem}
In the 1920's Cartan used a different way to associate a center to a finite set of points in a negatively curved Riemannian manifold, called the \emph{center of mass}. This definition could be adapted to our context as follows. By the definition of $r_{\max}$, for each $p \in \cY$ the function $\rho_p$ is strictly convex in {$E_{r_{max}}(\cY)$}. Thus, the function 
$\max_{p \in \cY}  \rho_p(\cdot)$ 
is strictly convex in $B_{r_{max}}^{\cap}(\cY)$, and achieves a unique minimum there. This minimum is defined to be  the center of mass $\textnormal{cm}(\cY)$. We point out that in general, the center of mass need not be the same as a center $c(\cY)$ we defined above, and does not serve our purposes.
\end{rem}

%===============================================================================

\subsection{Critical points for the distance function.}

In classical Morse theory, the critical points of a functions are those points where  the gradient $\nabla f$ is zero. The index of a critical point is the number of negative eigenvalues of the Hessian $H_f$, which can be thought of as the number of independent directions we can leave the critical point along which the function values will be decreasing. Consequently, critical points of index $0$ are the minima, and if $f$ is defined over a $d$-dimensional manifold, then critical points of index $d$ correspond to the maxima.
We now wish to investigate the critical locus of $\rho^2_{\cP}$ in an analog fashion. However, since $\rho^2_{\cP}$ is non differentiable, definitions of critical points have to be adjusted. 
We start by noticing that $\rho^2_{\cP}$ is nonnegative and vanishes precisely at the points $p \in \mathcal{P}$. Thus, $\cP$ is the set of minima, or index $0$ critical points, of $\rho^2_\cP$. 
To find the critical points of higher index, we will use the results of \cite{gershkovich_morse_1997}, which require the following definition.

%===============================================================================

\begin{defn}\label{def:delta}
Let $\cY = \lbrace y_1,\ldots, y_k \rbrace \subset M$ and $p \in M$, then we define the polytope in $\Delta(\cY) \subset T_{p}M$, given by
\[
	\Delta(\cY) := \conv(\set{\nabla \rho^2_{y_i}(p)}_{i=1}^k),
	\]
where $\conv(\cdot)$ means convex hull, and $\nabla \rho^2_{y_i}$ is the gradient of $\rho^2_{y_i}$, defined as the unique vector field such that $g(\nabla \rho^2_{y_i}, v)=d\rho^2_{y_i}(v)$, for any vector field $v$.
\end{defn} 

%===============================================================================

The critical points of $\rho^2_{\cP}$ are then characterized by the following Proposition.

%===============================================================================

\begin{prop}\label{prop:CriticalPoint}
Let $c \in M$ be a critical point of $\rho^2_\cP$. Then, $c$ has index $k$ if and only if there exists a set $\cY\subset\cP$ of $k+1$ points such that:
\begin{enumerate}
\item $c(\cY)=c$, where $c(\cY)$ is the center of $\cY$ (see Lemma \ref{lem:center}),
\item  $\Delta(\cY) \in T_cM$ contains the origin $0\in T_cM$,
%\item $\Delta(\cY)$ does not lie in any $(k-1)$-dimensional subspace of $T_{c}M$.
\item $B(\cY)\cap \cP = \cY$.
\end{enumerate}
\end{prop}

%===============================================================================
Figure \ref{fig:crit_pts} depicts the conditions in the last proposition.
%===============================================================================

\begin{proof}

{Let $c$ be a critical point} of the min-type function $\rho^2_\cP$ and suppose that the minimal representation of $\rho^2_\cP$ in a neighborhood of $c$ is of the form
$$\rho^2_\cP (\cdot)=\min_{i=1,...,k+1} \rho^2_{y_i}(\cdot).$$
Denoting $\cY = \set{y_1,\ldots, y_{k+1}}$, and given that the representation above is minimal, at $c$ we must have $\rho_{y_1}(c)=...=\rho_{y_{k+1}}( c)$.
In addition, using the definition of critical points of a min-type function from \cite{gershkovich_morse_1997}, we have that $c$ is a critical point of  each $\rho_{y_i}$ restricted to the smooth $d-(k-1)$ dimensional submanifold:
$$E(\cY) \cap B_{r_{\max}}(c) = \lbrace p \in B_{r_{\max}}(c)  \  \vert  \  \rho_{y_1}(p)=...=\rho_{y_{k+1}}(p) \rbrace.$$
However, by the definition of $r_{max}$, each distance function is strictly convex on $E_{r_{\max}}(\cY)$, and so has a unique critical point (a minimum) which is $c(\cY)$ by definition. Hence, by \cite{gershkovich_morse_1997} the index $I_c$ of $c=c(\cY)$ is
$$I_c(\rho) = k  + I_c(\rho \vert_{E(\cY)})=k.$$
The second and third conditions are now immediate from the definition of a critical point of a min-type function in \cite{gershkovich_morse_1997}. Conversely, if the conditions in the statement hold, then it immediately follows from the definition of a critical point of a Morse min type function in \cite{gershkovich_morse_1997} that $c$ is an index $k$ critical point.
\end{proof}

%===============================================================================

\begin{figure}[h]
\begin{center}
\includegraphics[scale=0.3]{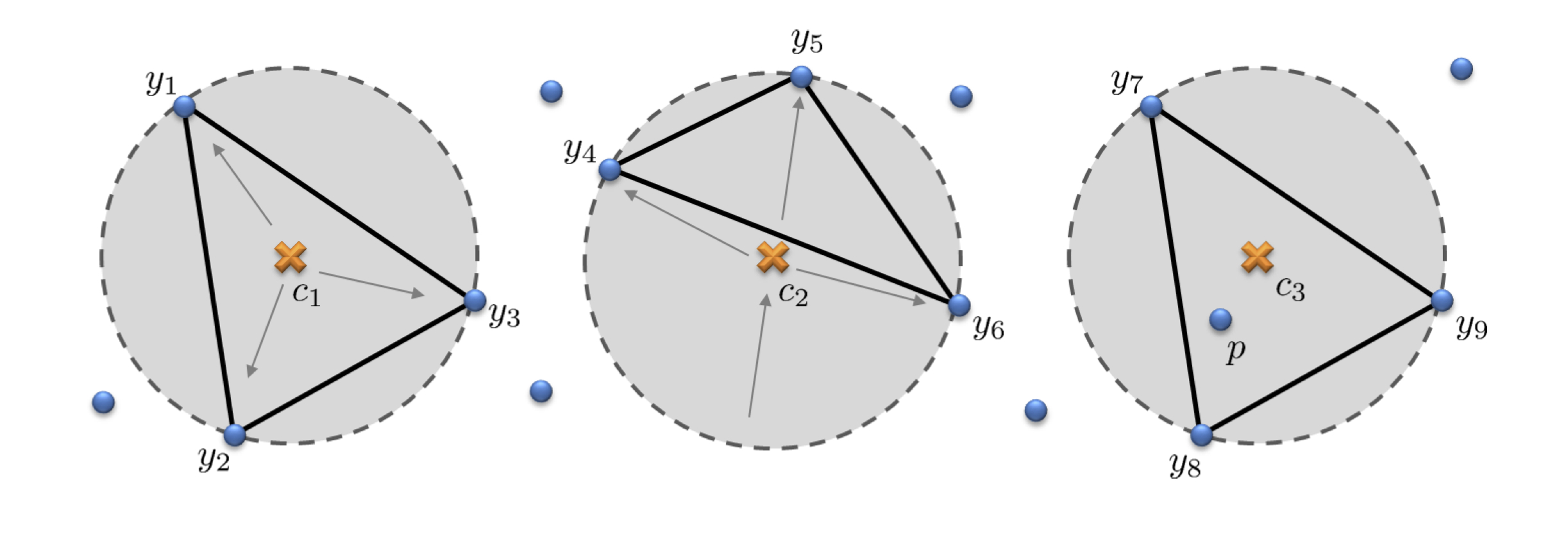}
\caption{\label{fig:crit_pts} Critical points of index $2$ in $\R^2$.
The blue points are the points of  $\cP$. We examine three subsets of $\cP$: $\cY_1 = \set{y_{_1},y_{_2},y_{_3}}$, $\cY_2 = \set{y_{_4},y_{_5},y_{_6}}$, and $\cY_3 = \set{y_{_7},y_{_8},y_{_9}}$. The orange x's are the centers  $c(\cY_i) = c_i$. The shaded balls are $B(\cY_i)$, and the interior of the triangles are $\Delta(\cY_i)$. The arrows represent the flow direction. For $\cY_1$ both conditions in Proposition \ref{prop:CriticalPoint} hold and therefore $c_1$ is a critical point. However, for $\cY_2$ condition (2) does not hold, and for $\cY_3$ condition (3) fails, therefore $c_2$ and $c_3$ are not critical points.}
\end{center}
\end{figure}

%===============================================================================

\begin{rem}
If $c \in M$ is a critical point of $\rho_{\cP}^2$ with $\cY \subset \cP$, and $c(\cY)=c$ as in proposition \ref{prop:CriticalPoint}. Then,  $c$ is not only a center $c(\cY)$ -  it is also the center of mass $\textnormal{cm}(\cY)$. However, this only true for critical points.
\end{rem}

%===============================================================================

\begin{rem}
In \cite{gershkovich_morse_1997}, the gradient of $\rho^2_\cP$ at a point $p \in M$ is defined to be $\Delta(\cY)$, where $\cY = \set{y_1,\ldots,y_k}$ forms a minimal representation of $\rho_\cP$ in a neighborhood of  $p$, i.e. $\rho^2_\cP (\cdot)= \min_i  \rho_{y_i}^2(\cdot)$. This definition of a critical point was first used (in an implicit way) in the work of Grove and Shiohama in \cite{Grove1977} and also in that of Gromov in \cite{Gromov1981}.
One could have equally defined $\Delta^*(\cY) \in T_p^*M$ as the convex hull of the $1$-forms $\lbrace d \rho_{y_i}(\cdot) \rbrace_{i=1}^{k+1}$. Then, the conditions $2$ and $3$ in the previous proposition have trivial analogues for $\Delta^*(\cY)$. We have chosen to work with $\Delta(\cY)$ rather than its dual in order to more closely follow the definitions in \cite{gershkovich_morse_1997}.
\end{rem}

%===============================================================================

Now that we have defined the critical points and their index, Morse theory (and in particular the inequalities discussed in Section \ref{sec:morse_ineq}) for $\rho_{\cP}^2$ follows from \cite{gershkovich_morse_1997}.
In particular, we can deduce the role of these critical points in generating the homology of the sublevel sets,
$$\rho_{\cP}^{-1}(0,r)=B_r(\cP),$$
which are homotopy equivalent to $\cC_r(\cP)$. 
While  we have considered $\rho_{\cP}^2$, since the sublevel sets of both $\rho_{\cP}$ and $\rho^2_{\cP}$ are the same union of balls (just at different levels), everything we discussed applies the same way to $\rho_{\cP}$.
Thus, from here own we shall refer about critical points and Morse theory for $\rho_{\cP}$, referring to the definitions we provided for $\rho^2_{\cP}$.

%===============================================================================

\subsection{Morse inequalities.}\label{sec:morse_ineq}

%===============================================================================

In proving the main results of this paper, we will make use Morse inequalities. In this section we present the version of these inequalities for Morse min-type functions which we shall apply to $\rho^2_{\cP}$. The proof uses a standard argument in algebraic topology and the reader unfamiliar with the basics of the subject is welcome to skip to remark \ref{rem:Morse_Inequalities}. 

Let $f$ be a min-type Morse function, and define:
\[
	C_k(a,b) := \# \text{ critical points $c$ of index $k$, such that $a < f(c) \le b$}
\]
The following lemma provides a slightly less standard version of the Morse inequalities.

%===============================================================================

\begin{lem}\label{lem:MorseInequalities}
Let $f:M \rightarrow \mathbb{R}$ be a Morse min-type function. For $a \in \mathbb{R}$ denote $M_a :=f^{-1}(-\infty, a]$. Then, for all $k \in \mathbb{N}$ the following inequalities hold 
\begin{equation}\label{eq:morse_ineq}
\beta_k(M_a)-\beta_k(M) \leq C_{k+1}(a, + \infty).
\end{equation}
\end{lem}

%===============================================================================

\begin{proof}
Using the approximation results of \cite{gershkovich_morse_1997} there is no loss of generality to restrict to the case when $f$ is a standard Morse-Smale function. Then, we consider the long exact sequence in homology for the pair $(M,M_a)$, namely
\[
\cdots\rightarrow H_{k+1}(M) \rightarrow H_{k+1}(M, M_a) \xrightarrow{\delta} H_k(M_a) \xrightarrow{i} H_{k}(M) \rightarrow \cdots
\]
The exactness of this sequence yields
\[
\splitb
\beta_k(M_a) & =  \dim H_{k}(M_a) = \dim( \im (i))+ \dim(\ker (i)) \\
& =  \dim( \im (i))+ \dim(\im (\delta)) \\
& \leq  \dim(H_k(M)) + \dim( H_{k+1}(M,M_a)).
\splite
\]
This proves that $\beta_k(M_a)- \beta_k(M) \leq  \beta_{k+1}(M,M_a)$. The statement then follows from the standard Morse inequalities applied to $f \vert_{M \backslash M_a}$, stating that $\beta_{k+1}(M,M_a) \leq C_{k+1}(a, + \infty)$.
\end{proof}

%===============================================================================

\begin{rem}\label{rem:Morse_Inequalities}
Applying Lemma \ref{lem:MorseInequalities} to the function $\rho_{\cP}^2$ we have $M_{r} = \rho_{\cP}^{-1}(-\infty , r] = B_{r} (\cP)$. Then, using the Nerve Lemma \ref{lem:nerve} we have $\beta_k(M_r)= \beta_k (\cC_r(\cP) )$ and so
\begin{equation}
\beta_k ( \cC_r (\cP) )- \beta_k(M) \leq C_{k+1}(r , + \infty),
\end{equation}
where $C_{k+1}(r, \infty)$ denotes the number of index $k+1$ critical points $c$ of $\rho_{\cP}^2$ with $\rho_{\cP}^2 (c) > r^2$.
\end{rem}

%===============================================================================

\section{Change of variables (Blaschke-Petkantschin-type Formula)}

%===============================================================================

To prove the main results in this paper we will need to evaluate probabilities related to the existence of certain critical points. These computations often result in complicated integral formulae. In this section we present a change-of-variables technique that simplifies these calculations significantly. This technique can be thought of as a Riemannian generalization of the Blaschke-Petkantschin formula that appeared in  \cite{edelsbrunner2016expected,miles1971isotropic}. We start by introducing some useful notation.

\begin{itemize}
\item $Gr(k,d)$ denotes the Grassmannian of $k$-planes in a $d$-dimensional real vector space. When we pick local coordinates, there is a fixed isomorphism $\mathbb{R}^d \cong T_cM$. Then, we may refer to $Gr(k,T_pM)$ when we want to emphasize the fact that we are parametrizing $k$-planes in the tangent space to $M$ at $p$. 

\item Pick normal coordinates $(x^1 , \ldots , x^d)$ centered at $p \in M$ associated with an orthonormal frame $\lbrace (\partial_{x^i})_c \rbrace_{i=1}^d$ of $T_pM$. For $V \in Gr(k,T_pM)$ and $r>0$, we define $S_r(V) \subset M$ to be the image under the exponential map of $(k-1)$-dimensional sphere of radius $r$ centered at the origin and equipped with the induced metric from $g$. In addition, we denote by $\mathbb{S}_r(V)$ the sphere of radius $r$ equipped with the Euclidean metric  $g_E = \delta_{ij} dx^i \otimes dx^j$. Notice that as a manifolds $S_r(V) = \mathbb{S}_r(V)$ by the Gauss lemma, the goal of the notation is to distinguish the metrics with which they are equipped.

\item We define $\ind_r:2^M\to\set{0,1}$ as,  
\[
\ind_r ( \cY) := \ind\set{ E_{r_{\max}}(\cY) \ne \emptyset \text{ and } \rho(\cY) \leq r},
\]
where $\ind\set{\cdot}$ stands for the indicator function.

\item Subsets of $M$ will be referred to as either $\cY\subset M$ or $\by \in M^k$, depending on the context. To keep notation simple, we will allow functions defined on $2^M$ (such as $c(\cdot),\rho(\cdot),\ind_r(\cdot)$), to be applied to $\by \in M^k$ as well.

\item Given $\cY = \lbrace y_1, \ldots , y_k \rbrace \subset M$ with well defined center $c=c(\cY)$ we  let  $v_i \in T_c M$ be such that $y_i = \exp_c(v_i)$, in particular $\abs{ v_i } = \rho(\cY)$. Let $(x^1, \ldots , x^d)$ be normal coordinates such that the $(\partial_{x^i})_c$ form an orthonormal frame of $(T_cM, \langle \cdot , \cdot \rangle)$. Suppose the $k$-vectors $v_i$ lie in a $(k-1)$-dimensional vector subspace $V \subset T_cM$ (this will be shown to be the case in the proof of Lemma \ref{lem:Integral}), then they generate a $(k-1)$-dimensional parallelogram. For $u \in \mathbb{R}^+$, let $\Upsilon_{u}(\textbf{v})$, for $\textbf{v}=(v_1, \ldots , v_k)$, be the $(k-1)$-volume of the parallelogram generated by the $u \frac{v_i}{\abs{ v_i }}$ in $T_c M$.
Moreover, notice that this is a homogeneous function of degree $(k-1)$ in $u$.
\end{itemize}

The following lemma will be  useful for us in the proofs of the main results.
The main idea is that instead of integrating over tuples $\by\in M^{k+1}$, we can perform a change of variables so that we first choose a $(k-1)$-sphere on which the points will be placed (with center $c$ and radius $u$), and then we place $k+1$ points on that sphere.

%===============================================================================

\begin{lem}\label{lem:Integral}
Let $\mathcal{P}$ and $r_{\max}$ be as in the discussion above, and let $r< r_{\max}$. Then, there is an invariant measure $\dmu_{k,d}$ on $Gr(k, d)$, such that for any smooth $f: M^{k+1} \rightarrow \mathbb{R}$, the integral $\int_{M^{k+1}} f(\by) \ind_r(\by) \abs{ \dvol_g(\by) }$ can be written as:
\begin{equation}\label{eq:big_int}
\begin{split}
 &\int_{M} \abs{ \dvol_g(c) }    \int_0^r du \ u^{dk-1}  \int_{Gr(k, T_{c}M) } \dmu_{k,d}(V)   \\  
&  \times \left(  \prod_{i=1}^{k+1} \int_{\mathbb{S}_1(V)}  \sqrt{\det(g_{ \exp_c(uw_i) } )} \abs{  \dvol_{\mathbb{S}_1(V)}(w_i) }\right) \ \Upsilon_1^{d-k}(\textbf{w}) f(\exp_c( u \textbf{w})) ,
\end{split}
\end{equation}
where $\textbf{w}=(w_1,...,w_{k+1}) \in (\mathbb{S}_1(V))^{k+1}$, $\Upsilon_1(\textbf{w})$ the $k$-volume of the parallelogram generated by the $\textbf{w}$ in $T_cM$, and $\dvol_{\mathbb{S}_{1}(V)}$ is the Riemannian measure of the Euclidean unit sphere in $V$.
\end{lem}

%===============================================================================

\begin{proof}
We start by noticing that since $r < r_{\max}$ the center $c=c(\by)$, as in Lemma \ref{lem:center}, is well-defined and such that the points in $\by$ are all at the same distance $u\le r$ from $c$. Hence, the points lie in a normal sphere centered at $c$, and can be written as $y_i = \exp_{c}(v_i)$, for some $v_i \in T_{c}M$ with $\abs{ v_i }=u$. We will show below that these $v_i$-s lie in a $k$-dimensional subspace $V \subset T_{c(\by)}M$. In this case, \eqref{eq:big_int} is  the result of integrating first over the center point $c$,  then over the possible distances $u=\rho(\by) \in [0,r]$, then over all possible $k$-dimensional subspaces $V$ in $T_{c}M$ containing the vectors $v_i$, and finally, over the $(k-1)$-spheres in $V$, where the $v_i$-s live. Next, we will justify \eqref{eq:big_int} by examining how the measures change under this reparametrization. 
 
Given the center $c = c(\by)$ and fixing normal coordinates $(x^1, \ldots , x^d)$ centered in $c$, we have
\begin{eqnarray}\nonumber
\dvol_g(\by) & = & \bigwedge_{i=1}^{k+1} \dvol_g(y_i) = \bigwedge_{i=1}^{k+1}  \sqrt{ \abs{ \det(g_{y_i}) } } \  dx^1(y_i) \wedge \ldots \wedge dx^d(y_i) \\ \nonumber
& = & \prod_{i=1}^{k+1} \sqrt{\abs{ \det(g_{y_i}) } } \ \bigwedge_{i=1}^{k+1}  dx^1(y_i) \wedge \ldots \wedge dx^d(y_i) .
\end{eqnarray}
Now let $u$ denote the distance between $y_i$ and $c$, then we have that $y_i=\exp_c(v_i)=\exp_c(u \frac{v_i}{\abs{ v_i }})$, for some $v_i \in T_cM$. In fact, in the coordinates $(x^1, \ldots , x^d)$ we have $v_i = \sum_{j=1}^d x^j(y_i) \left( \frac{\partial}{\partial x^i} \right)_c$. Then, using the Euclidean version of the Blaschke-Petkantschin formula \cite{miles1971isotropic} in the coordinates $(x^1, \ldots, x^d)$ we have
\[
\abs {\dvol_g(\by) }  = \left( \prod_{i=1}^{k+1} \sqrt{ \abs{ \det(g_{ \exp_c( v_i) } ) }} \right) \Upsilon_u^{d-k}(\bv) \  \abs {\dvol_g(c) } \ du \   d\mu_{k,d} (V) \abs{  \bigwedge_{i=1}^{k+1}   \dvol_{\mathbb{S}_u(V)}(v_i)  }  .
\]
Recall that $\Upsilon_u(\textbf{v})$ denotes the $k$-volume of the parallelogram spanned by the $v_i$-s in $V \subset T_cM$, measured using the Euclidean metric on $V$ induced by taking the $\partial_{x^i}$ to form an orthonormal frame on $T_cM$. Equivalently, if we choose the coordinates $(x^1, \ldots , x^d)$ so that $\lbrace ( \partial_{x^i} )_c \rbrace_{i=1}^d$ forms an orthonormal frame of $(T_cM, \langle \cdot , \cdot \rangle_c )$, then $\Upsilon_u(\textbf{v})$ is simply the $k$-volume of the parallelogram spanned by the $v_i$-s in $T_c M$. In particular, $\Upsilon_u$ is homogeneous of degree $k$ in $u$, so that $\Upsilon_u = u^{k} \Upsilon_1$. Similarly $\dvol_{\mathbb{S}_u(V)}$ is homogeneous of degree $k-1$, and therefore $\dvol_{\mathbb{S}_u(V)}= u^{k-1} \dvol_{\mathbb{S}_1(V)}$. Putting it all  together we have
\[
\abs{ \dvol_g(\by) } = \abs{ \dvol_g(c) }  u^{dk-1} du  \ d\mu_{k,d} (V)  \abs{ \bigwedge_{i=1}^{k+1} \sqrt{ \abs{ \det(g_{ \exp_c(v_i) } ) }}   \ \dvol_{\mathbb{S}_1(V)} \left( \frac{v_i}{\abs{ v_i }} \right) } \ \Upsilon_1^{d-k}(\textbf{v}),
\]
Finally, note that  $y_i=\exp_c(v_i)= \exp_c ( u \frac{v_i}{\abs{ v_i }} )$, and by setting $w_i =\frac{v_i}{\abs{ v_i }}$ we obtain \eqref{eq:big_int}.\\

To complete the proof, we need to show that  $v_1,\ldots, v_{k+1}$ lie in a $k$-dimensional subspace of $T_cM$.
We start by showing that $v_i =-u  (\nabla \rho_{y_i})_c$, where $u=\rho_{y_i}(c)= \abs{ v_i }$. This follows from the fact that by definition, $v_i$ is such that $\rho_{y_i}(\exp_{c}(s v_i))=u(1-s)$ for all $s \in [0,1]$. In particular $s v_i$ is the point on the sphere of radius $s u$ in $T_cM$, such that $\exp_c(s v_i)$ minimizes the distance to $y_i$. {Hence, for any $v' \in T_{s v_i} T_cM \cong T_cM$ with $\langle v' , v_i\rangle_c=0$ 
$$\langle  \nabla \rho_{y_i} , (d \exp_c )_{s v_i} (v') \rangle_{\exp_c(s v_i)}=  \frac{d}{dt} |_{t=0} \rho_{y_i}(\exp_c(s v_i + t v'))=0,$$
and similarly by differentiating $\rho_{y_i}(\exp_{c}(s v_i))=u(1-s)$ with respect to $s$
$$ \langle \nabla \rho_{y_i} , (d \exp_c )_{s v_i} v_i \rangle_{\exp_c(s v_i)} = -u.$$}
The Gauss Lemma guarantees that $\abs{ d \exp_{sv_i} v_i }= \abs{ v_i } =u$ and so we conclude that for $s \neq 0$, $(\nabla \rho_{y_i})_{\exp_c(sv_i)}= -u^{-1} (d \exp_c )_{s v_i} v_i$. The continuity of $\nabla \rho_{y_i}$ in a normal ball and the fact that $(d \exp_c)_0$ is the identity then implies that $v_i=-u \nabla \rho_{y_i}$ at $c$, as we wanted to show.
Now we simply need to show that the $\nabla \rho_{y_i}$'s for $i \in \lbrace 1, ..., k+1 \rbrace$ are linearly dependent at $c$. Arguing by contradiction, suppose they are linearly independent and so span a $(k+1)$-dimensional space $W \subset T_{c(\by)}M$. Let $W= \text{span} \lbrace \nabla \rho_{y_i} \rbrace_{i=1}^{k+1}$, for a generic set $\by \in M^{k+1}$, $E(\by) \cap \exp_{c}(W) \subset E(\by)$ { is a nonempty $1$-dimensional manifold }(it is nonempty since $c \in E(\by) \cap \exp_{c}(W)$). By definition $c \in E(\by) \cap \exp_{c}(W)$ is a critical point of all the $\rho_{y_i}$'s restricted to $E(\by)$ and so { the orthogonal projection of $(\nabla \rho_{y_i})_c$'s to $T_c (E(\by) \cap \exp_{c}(W))$ must vanish}. However,  by definition $\lbrace \nabla \rho_{y_i} \rbrace_{i=1}^{k+1}$ span the tangent space to $W$, and therefore we have obtained a contradiction with the fact that $E(\by) \cap \exp_{c}(W)$ is a nonempty $1$-dimensional submanifold.
\end{proof}

%===============================================================================

\begin{rem}
In the previous proof, the fact that the $w_i$-s lie in a $k$-dimensional subspace of $T_cM$ can be intuitively explained as follows. If this is not true, then the image under the exponential map of a small ball around $0$ in $W$ is a smooth $(k+1)$-dimensional manifold which intersects $E(\by)$ along a $1$-dimensional manifold. Moreover, there is a direction along this intersection such that the function $\rho_{y_i}(\cdot)$ is decreasing, and so contradicting the fact that $c(\by)$ is the minimum.
\end{rem}

%===============================================================================

\section{Expected Betti Numbers - Upper Bound}\label{sec:upper}

%===============================================================================

In this section we present an upper bound on the Betti numbers of a random \cech complex in terms of $\Lambda$ (recall that $\Lambda = n\omega_d r^d$  is approximately the expected number of points inside a ball of radius $r$).
This upper bound is interesting by itself, and also useful for finding the upper threshold in Theorem \ref{thm:main}. The main result in this section is the following.

%===============================================================================

\begin{prop}\label{prop:betti_order}
Let $n \to \infty$ and $r \to 0$ in such a way that $\Lambda \to  \infty$ and $\Lambda r  \to 0$ 
%\footnote{In fact, all we will use in the proof is the weaker assumption that $\Lambda r^2 \abs{ \log r }^{(d+2)/d} \to 0$. This is still not the optimal rate for the conclusion to still be true so we prefered to state it with the stronger assumption that just $\Lambda r \to 0$, which we do need in the Theorem \ref{thm:main} due to the second moment computations in Proposition \ref{prop:main}. }. 
Then, for every $1\le k\le  d-1$ there exists a positive constant $b_{k}>0$ (depending on $k$, $d$, and $g$) such that
$$\meanx{\beta_k(r))} \le \beta_k(M) + b_{k}{n\Lambda^k e^{-\Lambda}}.$$
\end{prop}

%===============================================================================

To prove Proposition \ref{prop:betti_order} we will use Morse theory discussed in Section \ref{sec:Morse}. The main idea is to use the Morse inequalities \eqref{eq:morse_ineq} and  bound  the number of critical points of $\rho_{\cP_n}$. Defining,
\[
	C_k(r_1, r_2) := \# \text{ critical points $c$ of $\rho_{\cP_n}$ with index $k$, such that $\rho_{\cP_n}(c) \in (r_1,r_2]$},
\]
we prove the following lemma. 

%===============================================================================

\begin{lem}\label{lem:crit_pts}
Let $n\to \infty$ and $r,r_0\to 0$ such that $r =o(r_0)$, $\Lambda\to \infty$, and $\Lambda_{r_0} r_0^{2}\to 0$, where $\Lambda_{r_0} := \omega_d nr_0^d$.
Then for every $k\ge 1$ we have
\[
\mean{C_{k}(r, r_0)}= O(n \Lambda^{k-1} e^{-\Lambda} ).
\]
\end{lem}

%===============================================================================

\begin{rem}\label{rem:const}
The proofs will make use of various constant values. Some of them will be given a name, while  the ones  whose value is not relevant for the main results will be denoted by $C>0$ (which might depend on $k,d,g$).
\end{rem}

\begin{proof}

Let $c$ be a critical point of the distance function $\rho_{\cP_n}$ with $\rho_{\cP_n}(c) \in (r, r_0]$.  Following the discussion in Section \ref{sec:Preliminaries} and Proposition \ref{prop:CriticalPoint}, we know that $c$ is generated by a subset $\cY\subset\cP_n$, so that $c = c(\cY)$, and $\cY$ satisfies the following:
\begin{equation}\label{eq:crit_pt_cond}
\textrm{(1) } 0 \in \Delta(\cY) \subset T_cM ,\qquad \textrm{(2) } {B(\cY)} \cap \cP = \emptyset,\qquad \textrm{(3) }  r < \rho(\cY)\leq r_0,
\end{equation}
where $\Delta(\cY)$ is defined in  \ref{def:delta}, and $B(\cY) := B_{\rho(\cY)}(c(\cY))$. Note that in this case we have $\rho_{\cP_n}(c) = \rho(\cY)$. Next, we define the following indicator functions:
\begin{equation}\label{eq:h_funcs}
\begin{split}
	h(\cY) &:= \ind\set{ 0 \in \Delta(\cY)},\\
	h_{r, r_0}(\cY) &:= h(\cY)\ind\set{r < \rho(\cY) \le r_0},\textrm{ and }\\
	g_{r, r_0}(\cY, \cP) &:= h_{r,r_0}(\cY)\ind\set{B(\cY)\cap (\cP\backslash \cY) = \emptyset}.
\end{split}
\end{equation}
With these definitions we can now write
\[
	C_k(r, r_0) = \sum_{\substack{\cY\subset\cP_n\\\abs{\cY} = k+1}} g_{r, r_0}(\cY ,\cP_n).
\]
Applying Palm theory to the mean value (see Theorem \ref{thm:prelim:palm}  in the appendix) we have that
\[
\mean{C_k(r, r_0)} = \frac{n^{k+1}}{(k+1)!} \mean{g_{r, r_0}(\cY' ,\cY' \cup\cP_n)},
\]
where $\cY'$ is a set of $(k+1)$ $\iid$ points, uniformly distributed on $M$, and independent of $\cP_n$. Next, the properties of the Poisson process $\cP_n$ imply that
\[
	\cmean{\ind\set{B(\cY')\cap \cP_n = \emptyset}}{\cY'} = \cprob{\cP_n(B(\cY'))=0}{\cY'}= e^{-n\vol(B(\cY'))} .
\]
Therefore, 
\[
\mean{C_k(r, r_0)} = \frac{n^{k+1}}{(k+1)!}\int_{M^{k+1}} h_{r, r_0}(\by) e^{-n\vol(B(\by))} \abs{ \dvol_g(\by)}.
\]

We can now apply Lemma \ref{lem:Integral} with $f(\by)=h(\by) e^{-n\vol(B(\by))} $, and have
\[
\begin{split}
\mean{C_k(r,r_0)} = \frac{n^{k+1}}{(k+1)!}&\int_{M} \abs{ \dvol_g(c) } \int_r^{r_0} du \ u^{dk-1}  \int_{Gr(k, T_{c}M) } \dmu_{k,d}(V) \\
			& \times \left(  \prod_{i=1}^{k+1} \int_{\mathbb{S}_1(V)}  \sqrt{ \abs{ \det(g_{  \exp_c(uw_i) } ) } } \  \abs{ \dvol_{\mathbb{S}_1(V)}(w_i)  } \right) \Upsilon_1^{d-k}(\textbf{w}) f(\exp_c( u \textbf{w})) ,
\end{split}
\]
where $c =c(\by)$, $u = \rho(\by)$, and $\by=\exp_c(u \textbf{w})$. Replacing $f(\by)$ with $h(\exp_c(u \textbf{w}))e^{-n\vol(B_u(c))}$, we have
\[
\begin{split}
\mean{C_k(r,r_0)} = \frac{n^{k+1}}{(k+1)!}&\int_{M} \abs{ \dvol_g(c) } \int_r^{r_0} du \ u^{dk-1}e^{-n \vol(B_u(c))}  \int_{Gr(k, T_{c}M) } \dmu_{k,d}(V) \\
			& \times \left(  \prod_{i=1}^{k+1} \int_{\mathbb{S}_1(V)}  \sqrt{ \abs{ \det(g_{  \exp_c(u w_i) } ) }} \ \abs{  \dvol_{\mathbb{S}_1(V)}(w_i) } \right) \Upsilon_1^{d-k}(\textbf{w}) h(\exp_c( u \textbf{w})).
\end{split}
\]
As the Grassmannian $Gr(k,T_cM)$ is compact, there is a subspace $V_{\max} \subset T_cM$ which maximizes the last integral over $(\mathbb{S}_1(V))^{k+1}$ and we have
\begin{equation}\label{eq:C_k_mean_intermediate1}
\splitb
\mean{C_k(r, r_0)} \leq & C {n^{k+1}} \int_{M} \abs{ \dvol_g(c) } \int_r^{r_0} du \ u^{dk-1} e^{-n\vol(B_u(c))} \\ 
& \times \left(  \prod_{i=1}^{k+1} \int_{\mathbb{S}_1(V_{max})}  \sqrt{ \abs{ \det(g_{ \exp_c(uw_i) } ) }} \abs{  \dvol_{\mathbb{S}_1(V_{max})}(w_i)} \right)  \Upsilon_1^{d-k}(\textbf{w}) h(\exp_c( u \textbf{w})),
\splite
\end{equation}
where $C> 0$ only  depends on the dimension $d$ and the index $k$. Using  \eqref{eq:RiemannianMeasure} and Lemma \ref{lem:approx_sphere_vol}, for each $y_i = \exp_c(u w_i)$, we can write 
\[
\sqrt{ \abs{ \det(g_{y_i} ) } } =  1 - \frac{Ric_{m n}}{3}x^m(y_i) x^n(y_i) + O(u^3) ,
\]
The second-order term above is bounded by $\frac{1}{3} \abs{ Ric } r_0^2$ (in fact only by the values of $Ric$ restricted to $V_{max}$). In addition, $\Upsilon_1(\bw)$ (the $k$-volume of the parallelogram generated by the $w_i$) is  bounded from above, since $\bw$ contains unit vectors. Putting it all back into \eqref{eq:C_k_mean_intermediate1} yields
\[
\splitb
\mean{C_k(r, r_0)}  \leq & C{n^{k+1}} (1+ c_{_R} r_0^2)^{k+1}  \int_{M} \abs{ \dvol_g(c) } \int_r^{r_0} du \ u^{dk-1} e^{-n\vol(B_{u}(c))}\\
 & \times \int_{(\mathbb{S}_1)^{k+1}} \abs{ \dvol_{{\mathbb{S}}_1}(\bv) } h(\exp_c( \bv)) \\ 
 = &  C n^{k+1}\int_{M} \abs{ \dvol_g(c) } \int_r^{r_0} du \ u^{dk-1} e^{-n\vol(B_{u}(c))} ,
\splite
\]
where $c_R$ depends on the metric $g$. The constant in $C$ in the last line includes the product of $(1+ c_{_R} r_0^2)^{k+1}$ with integral of $h$ over the $(k+1)$ spheres. 
Recall from  \eqref{eq:ball_vol} and that for small $u$ we have 
$$\vol(B_{u}(c))= \omega_d u^d \left( 1-\frac{s(c)}{6(d+2)} u^2 + O(u^3) \right).$$
Thus, using the Taylor expansion $e^{ x}=1+ x +O(x^2)$,  and the fact that $u\le r_0 \to 0$, we have
\begin{equation}\label{eq:volume_estimate_normal_coords}
\splitb
e^{-n\vol(B_u(c))} &= e^{-n\omega_d u^d(1- \frac{s(c)}{6(d+2)}u^2 + o(u^2))} \\
&= e^{-n\omega_d u^d} \param{1+ \frac{s(c)}{6(d+2)}n\omega_du^{d+2} + o(n u^{d+2}) } \\
&\le e^{-n\omega_d u^d} \param{1+ s_{\max} n\omega_dr_0^{d+2}} 
\splite
\end{equation}
where $s_{\max}= \sup_{c \in M} \frac{s(c)}{6(d+2)}+\delta $ for some $\delta>0$.  
Applying the change of variables $s=\frac{u}{r}$, and recalling that $\Lambda = n\omega_d r^d$, yields
\begin{equation}\label{eq:C_k_ineq_1}
\splitb
\mean{C_k(r, r_0)} & \leq   C n^{k+1}  r^{dk}  \param{1+ s_{\max} n\omega_dr_0^{d+2}} \int_{M} \abs{ \dvol_g(c) } \int_1^{\frac{r_0}{r}} ds \ s^{dk-1} e^{-\Lambda s^d} ,\\
&= C n \Lambda^k\param{1+ s_{\max} \Lambda_{r_0}r_0^{2}} \int_{M} \abs{ \dvol_g(c) } \int_1^{\frac{r_0}{r}} ds \ s^{dk-1} e^{-\Lambda s^d}.
\splite
\end{equation}
 The last integral is known as the \emph{lower incomplete gamma function} and has a closed form expression which yields,
\begin{eqnarray}\nonumber
\mean{C_k(r, r_0)} & \le & n C (1+ s_{\max} \Lambda_{r_0}r_0^{2}) \param{ \left( 1-e^{-\Lambda_{r_0}}\sum_{j=0}^{k-1}\frac{\Lambda_{r_0}^j}{j!} \right) - \left( 1-e^{-\Lambda}\sum_{j=0}^{k-1}\frac{\Lambda^j}{j!} \right)  } \\ 
& = & n  C (1+ s_{\max}\Lambda_{r_0}r_0^{2}) \param{ e^{-\Lambda}\sum_{j=0}^{k-1}\frac{\Lambda^j}{j!} - e^{-\Lambda_{r_0}}\sum_{j=0}^{k-1}\frac{\Lambda_{r_0}^j}{j!}   }.\label{eq:crit_pts_final}
\end{eqnarray}

 Finally, using the  assumptions that $\Lambda_{r_0}r_0^{2} \rightarrow 0$,  $\Lambda \rightarrow  \infty$, and  $r =o(r_0)$ yields $\mean{C_{k}(r, r_0)}= O(n \Lambda^{k-1} e^{-\Lambda} )$, and that completes the proof.

\end{proof}

%===============================================================================

We are now ready to prove Proposition \ref{prop:betti_order}.

%===============================================================================

\begin{proof}[Proof of Proposition \ref{prop:betti_order}]

Let $\hat\beta_k(r) := \beta_k(r)-\beta_k(M)$, then we need to show that $\meanx{\hat\beta_k(r)}\le b_{k} n\Lambda^k e^{-\Lambda}$ for some $b_{k}>0$. 
Using Morse theory, and in particular  Lemma \ref{lem:MorseInequalities}, it is enough to control the number of critical points of index $k+1$ occurring at a radius greater than $r$ (i.e.~$C_{k+1}(r, +\infty)$).

Let $r < r_0 < r_{\max}$  such that  $r=o(r_0)$, and $\Lambda_{r_0}r_0^{2}\to 0$. Let $E$ denote the event that $B_{r_0}(\cP_n)$ covers $M$, then
\begin{equation}\label{eq:mean_beta_k_cond}
	\meanx{\hat\beta_k(r)} = \cmeanx{\hat\beta_k(r)}{E}\prob{E} + \cmeanx{\hat\beta_k(r)}{E^c}\prob{E^c}.
\end{equation}

The proof will now be split into two steps, dealing with evaluating each of the terms in the sum above.

\noindent\textbf{Step 1}: We prove that: $\cmeanx{\hat\beta_k(r)}{E}\prob{E} = O(n \Lambda^k e^{-\Lambda}).$

If the event $E$ occurs, then $\rho_{\cP_n}(x)\le r_0$ everywhere on $M$, and therefore Lemma \ref{lem:MorseInequalities} applied to $\rho_{\cP_n}$ implies that any nonvanishing $k$-cycle in $\cC_r$ which is mapped to a trivial cycle in $M$, is terminated by a critical point of index $k+1$ with value  in $(r,r_0]$.
Therefore, we must have $\hat\beta_k(r) \le  C_{k+1}(r, r_0)$, and then
\eqb \label{eq:beta_bound_crit}
\cmeanx{\hat\beta_k(r)}{E}\prob{E} \le  \cmeanx{ C_{k+1}(r, r_0)}{E}\prob{E} 
\le\meanx{ C_{k+1}(r, r_0)}.
\eqe

Since $\Lambda\to \infty$ and $\Lambda_{r_0}r_0^{2}\to 0$, using Lemma \ref{lem:crit_pts} we have that $\mean{C_{k+1}(r, r_0)}= O(n \Lambda^{k} e^{-\Lambda} )$, and that completes the first step.

\noindent\textbf{Step 2:} We prove that: $\cmeanx{\hat\beta_k(r)}{E^c}\prob{E^c} = o(n \Lambda^k e^{-\Lambda}).$
%We prove that the second summand in \eqref{eq:mean_beta_k_cond} is $o( n e^{-\Lambda}\Lambda^k)$. 

For any simplicial complex, the $k$-th Betti number is bounded by the number of $k$-dimensional faces (see, for example, \cite{hatcher_algebraic_2002}). Since the number of faces is bounded by $\binom{\abs{\cP_n}}{k+1}$, we have that 
\begin{equation}\label{eq:mean_betak_no_cover}\begin{split}
\cmeanx{ \beta_k(r)}{E^c} \prob{E^c}&\le \cmean{\binom{ \abs{ \cP_n }}{k+1}}{E^c}\prob{E^c} \\
	&= \sum_{m=k+1}^\infty \binom{m}{k+1}\cprob{\abs{\cP_n} = m}{E^c}\prob{E^c} \\
	&=\sum_{m=k+1}^\infty \binom{m}{k+1}\cprob{E^c}{\abs{\cP_n}=m}\prob{\abs{\cP_n} = m},
	\end{split}
\end{equation}
where we used Bayes' Theorem. Since $\cP_n$ is a homogeneous Poisson process with intensity $n$ we have that $\prob{\abs{\cP_n}=m} = \frac{e^{-n}n^m}{m!}$, and also that given $\abs{\cP_n}=m$ we can write $\cP_n$ as a set of $m$ $\iid$ random variables $\cX_m = \set{X_1,\ldots, X_m}$ uniformly distributed on $M$. Therefore,
\[
	\cprob{E^c}{\abs{\cP_n}=m} = \prob{ B_{r_0}(\cX_m) \ne M}.
\]
Next, we will bound this coverage probability.
Let $\cS$ be a $\frac{r_0}{2}$-net of $M$, i.e.~for every $x \in M$ there is a point $s\in \cS$ such that $\dist(x,s)\le \frac{r_0}{2}$.
                                                                                                                                                                                                                                                                                                                                                                                                                                                                                                                                                                                                                                                                                                                                                                                                                                                                                                                                                                                                                                                                                                                                                                                                                                                                                                                                                                                                                                                                                                                                                                                                                                                                                                                                                                                                                                                                                                                                                                                                                                                                                                                                                                                                                                                                                                                                                                                                                          We can find such a $\frac{r_0}{2}$-net with $\abs{\cS} \le c_d r_0^{-d}$, where $c_d$ is a constant that depends only on $d$ and the metric $g$. Note that if for every $s\in\cS$ there exists  $X_i\in \cX_m$ with $\rho(s,X_i)\le \frac{r_0}{2}$, then for every $x\in M$ we have
\[
\dist(x,X_i) \le \dist(x,s) + \dist(s,X_i) \le r_0,
\]
and therefore $B_{r_0}(\cX_m) = M$. Thus, if $B_{r_0}(\cX_m) \ne M$ then there exists $s\in \cS$ with $\rho_{_{\cX_m}}(s) > \frac{r_0}{2}$, which yields
\[
\prob{ B_{r_0}(\cX_m) \ne M )}\le \sum_{s\in\cS}\prob{\rho_{\cX_m}(s) >\frac{r_0}{2}} \le c_d r_0^{-d} (1-\alpha r_0^d )^m ,
\]
for any $\alpha \leq  2^{-d}\omega_d \left(1- s_{max} r_0^2 \right)$, using Corollary \ref{cor:Volume}. Putting it all back into \eqref{eq:mean_betak_no_cover} we have
\[
\splitb
\cmeanx{\beta_k(r)}{E^c} \prob{E^c} & \le  \sum_{m=k+1}^\infty \binom{m}{k+1} c_d r_0^{-d} (1-\alpha r_0^d)^m \frac{e^{-n}{n^m}}{m!} \\
& =  C r_0^{-d}(n(1-\alpha r_0^d))^{k+1} e^{-\alpha nr_0^d} \\
&\le C r_0^{-d} n^{k+1} e^{-\alpha nr_0^d}.
\splite
\]
To finish the proof, we  take $r_0 = r \left( \frac{\omega_d}{\alpha} \left( 1 + \abs{ \log r  } \right) \right)^{1/d}$,
and use the assumption that $\Lambda r \to 0$. In particular, one can show that: (a) $r=o(r_0)$, (b) $\Lambda_{r_0}r_0^{2}\to 0$, and (c) $r_0^{-d} n^{k+1} e^{-\alpha nr_0^d} = o(n e^{-\Lambda}\Lambda^k)$.
Therefore,
\begin{equation}\label{eq:bound_bk_Ec}
\cmeanx{\beta_k(r)}{E^c} \prob{E^c}=  o(n e^{-\Lambda}\Lambda^k).
\end{equation}
Finally, recall that $\hat\beta_k(r) = \beta_k(r) -\beta_k(M)$. Note that $\cmean{\beta_k(M)}{E^c}= \beta_k(M)$, and in addition, similar calculations to the ones above yield $\prob{E^c} = O(e^{-\alpha n r_0^d}) = o(n\Lambda^k e^{-\Lambda})$.
These facts together with \eqref{eq:bound_bk_Ec} show that $\cmeanx{\hat\beta_k(r)}{E^c} \prob{E^c}=  o(n e^{-\Lambda}\Lambda^k)$, and conclude the proof.\\
\end{proof}

%===============================================================================

\begin{rem}
Notice that if the scalar curvature is everywhere negative and  instead of requiring that $n r^{d+2} \rightarrow 0$ take $n r^{d+3} \rightarrow 0$, then the above bound gets smaller, namely:
$$\meanx{\beta_k(r)} \le \beta_k(M) + b_{k}(1+s_{\max} \Lambda_{r_0} r_0^2){n\Lambda^k e^{-\Lambda}},$$
with $s_{\max}= \delta + \sup_{c \in M} \frac{s(c)}{6(d+2)}<0$ (i.e. we choose $\delta$ so that $s_{\max}<0$). It is therefore conceivable that the scalar curvature affects some lower-order term the phase transition of Theorem \ref{thm:main}.
\end{rem}

%===============================================================================

\section{Expected Betti Numbers - Lower Bound}\label{sec:lower}

%===============================================================================

In this section we compute a lower bound for the Betti numbers of $\cC_r$ in terms of $\Lambda$. While in Section \ref{sec:upper} we obtained an upper bound by making use of the Morse inequalities to simply count critical points, in order to obtain a lower bound we must proceed differently, since even if there are many critical points it is not clear which homology degree they contribute to.
Thus, we shall instead consider a special type of critical points that are guaranteed to generate nontrivial cycles. These were first introduced in \cite{bobrowski_vanishing_2015} and named $\Theta$-cycles, after their unique structure. We start by stating the main result in this section.

%===============================================================================

\begin{prop}\label{prop:betti_order2}
Let $n\to\infty$ and $r\to 0$ such that  $\Lambda\to \infty$ and $\Lambda r^{2} \to 0$, then for every $1\le k\le  d-1$ there exists $a_{k}>0$ (depending only on $k,d$ and the metric $g$) such that
\[
\meanx{\beta_k(r)} \ge a_{k} n\Lambda^{k-2}e^{-\Lambda}.
\]
Moreover, if $(M,g)$ has everywhere positive scalar curvature, one just needs to assume that $nr^{d+2}$ stays bounded.
\end{prop}

%===============================================================================

As mentioned earlier, the proof uses the strategy of \cite{bobrowski_vanishing_2015}, and follows from combining the following lemmas. We start with some definitions. Let $\mathcal{P} \subset M$ and let $\mathcal{Y} \subset \mathcal{P}$ be a generic set. For $\alpha>0$  define the closed annulus
$$A_{\alpha}(\cY)= { B_{\rho(\cY)}(c(\cY)) } \backslash B^\circ_{\alpha \rho(\cY)}(c(\cY)),$$
where $B_r^\circ(p)$ is an open ball. Notice that $\alpha>0$ represents a scale invariant quantity. 
The following lemma is taken from \cite{bobrowski_vanishing_2015} where it is proved for the torus. However, the proof remains the same for any compact Riemmanian manifold $M$.

%===============================================================================

\begin{lem}[Lemma 7.1 in \cite{bobrowski_vanishing_2015}]\label{lem:theta_cycle}
Let $\mathcal{P} \subset M$, and let $\mathcal{Y} \subset \mathcal{P}$ be a set of  $k+1$ points, such that $c(\cY)$ is a critical point of index $k$. Define
$$\phi = \phi(\cY) := \frac{1}{2 \rho(\cY)} \min_{v \in \partial \Delta(\cY)} \abs{ v } ,$$
If $\rho(\cY) < r_{\max}$ and $A_{\phi}(\cY) \subset B_{\rho(\cY)}(\cP)$, then the critical point $c(\cY)$ generates a new nontrivial cycle in $H_{k}(B_{\rho(\cY)}(\cP))$. 
%In particular, the $k$-simplex generated by $c(\cY)$ is isolated, i.e. not the face of any $(k+1)$-simplex.
\end{lem}

%===============================================================================

The cycles created this way are called $\Theta$-cycles, and the the idea behind the proof of Proposition \ref{prop:betti_order2} is to count them. Let $\eps>0$, and define $\beta^{\eps}_k(r)$ to be the number of $\Theta$-cycles generated by those $\mathcal{Y}$, such that
\[
\text{(C1)}\ \rho(\cY) \in (r_1,r] ,\quad \text{(C2)}\ B_{r_2}(c(\cY)) \cap \cP = \cY, \quad\text{and} \quad \text{(C3)}\ \phi(\cY) \ge \eps,
\]
where $r_2>r>r_1>0$ are positive real constants (to be determined later).
The next lemma shows that indeed, the $\Theta$-cycles counted by $\beta^{\eps}_k$ provide a lower bound on the Betti numbers. In the statement of the next result, the constant $c_g$ is taken from Lemma \ref{lem:intersect}.

%===============================================================================

\begin{lem}\label{lem:Betas}
Let $r, r_2 \in \mathbb{R}^+$ with $r_2 > r$, then for $r_1 > r \sqrt{ 1 -\frac{1}{c_g^2} \left( \frac{r_2}{r} -1 \right)^2 }$ and any $\eps \in (0,1)$, we have $\beta_k(r) \geq \beta^{\eps}_k (r)$.
\end{lem}

%===============================================================================

\begin{proof}
We need to show that any such $\Theta$-cycle created prior to $r$ still exists at $r$ and so gives rise to a nonzero element in $H_k(\cC_r)$, i.e. it contributes to $\beta_k(r)$.

The proof of Lemma \ref{lem:theta_cycle} uses the fact that every $\Theta$-cycle introduces an uncovered $k$-simplex $\Delta$  in $\cC_{r_1}$ (i.e.~$\Delta$ is not a face of any $(k+1)$-simplex). Since uncovered simplexes cannot be boundaries, it is thus enough to show that $\Delta$ is still uncovered at radius $r$. This requires that $B_r^{\cap}(\cY)$, does not intersect any of the balls $B_{r}(p)$, for $p \in \cP\backslash \cY$. Using Lemma \ref{lem:intersect}, and condition (C1) above, we have that for all $x \in B_r^\cap(\cY)$,
$$\dist( c(\cY) , x) \leq c_g \sqrt{r^2-\rho(\cY)^2} \le c_g \sqrt{r^2-r_1^2}.$$
Using condition (C2), for all $p\in \cP\backslash \cY$ we have $\dist(p,c(\cY)) \geq r_2$. Thus, using the triangle inequality we have
$$\dist(p, x) \geq \dist(p, c(\cY)) - \dist(c(\cY),x) \geq r_2 - c_g \sqrt{r^2-r_1^2}.$$
Hence, if  we take  $r_1 > r \sqrt{ 1 -\frac{1}{c_g^2} \left( \frac{r_2}{r} -1 \right)^2 }$, we get that $\dist(p, x) > r$, which implies that $B_r(\cP \backslash \cY) \cap B_r^\cap(\cY) = \emptyset$. Thus, $\Delta$ is still uncovered at radius $r$, which completes the proof.\\
\end{proof}

%===============================================================================

Next, we define the following (related to conditions (C1)-(C3) above).
\[
\begin{split}
	h_r^{\eps}(\cY) &:= h_{r_1,r}(\cY)\ind\set{\phi(\cY) \ge \eps},\\
	g_r^\eps(\cY,\cP) & := h_r^\eps(\cY,\cP)\ind\set{B_{r_2}(c(\cY))\cap (\cP\backslash \cY) = \emptyset}\ind\set{A_{\eps}(\cY) \subset B_{\rho(\cY)}(\cP )}.
	\end{split}
\]
Thus, we can write
\begin{equation}\label{eq:Beta_k_Epsilon}
	\beta_k^\eps(r) = \sum_{\cY\subset \cP_n} g_r^\eps(\cY,\cP_n).
\end{equation}

The last result we need before proving Proposition \ref{prop:betti_order2} is the following lemma.

%===============================================================================

\begin{lem}\label{lem:theta_cycle_order}
Let $\eps >0$ be sufficiently small, and let $r>0$ be such that $\Lambda\to \infty$, and $\Lambda r^{2}\to 0$. Then there exists $a_k>0$, and there exists a choice of $r_1,r_2$ with $0<r_1<r<r_2$, such that 
\[
	\meanx{\beta_k^\eps(r)} \geq a_k n\Lambda^{k-2}e^{-\Lambda}.
\]
\end{lem}

%===============================================================================

\begin{proof}
Let $\eps>0$, then the expectation of $\beta_k^{\eps}(r)$, can be computed in a similar way to the computation in the proof of Proposition \ref{prop:betti_order}. Suppose that  $r_2 < r_{\max}$ (defined in Section \ref{sec:Preliminaries}), then using  Palm theory and the properties of Poisson processes we have (Theorem \ref{thm:prelim:palm}),
\begin{equation}\label{eq:mean_C_k_th}
\begin{split}
	\meanx{\beta_k^{\eps}(r)} &= \frac{n^{k+1}}{(k+1)!}\mean{g_r^\eps(\cY', \cY'\cup\cP_n)}\\
	&= \frac{n^{k+1}}{(k+1)!}  \int_{M^{k+1}} h_r^{\eps}(\by) p_{\eps}(\by) e^{-n \Vol (B_{r_2}(c(\by))) } \abs{ \dvol_g(\by) },
\end{split}
\end{equation}
where
\[
p_{\eps}(\by) := \cprob{A_\eps(\cY') \subset B_{\rho(\cY')}(\cP_n\cup\cY')}{ \cY'=\by, \cP_n\cap B_{r_2}(c(\cY')) = \emptyset }.
\]
We shall now evaluate the integral in \eqref{eq:mean_C_k_th} in two steps.

\noindent\textbf{Step 1:} We show that $p_\eps \to 1$ uniformly in $\by$ as $n\to \infty$. Denoting $\P_{_\emptyset}(\cdot) := \cprob{\cdot}{\cP_n\cap B_{r_2}(c(\by)) = \emptyset}$,   we have that
\[ 
p_{\eps}(\by) \ge\P_{\emptyset}\param{A_\eps(\by) \subset B_{\rho(\by)}(\cP_n)}.
\]
In the following we will use the shorthand notation:
\[
	\rho = \rho(\by),\quad A_\eps = A_\eps(\by), \quad B_{r_2} = B_{r_2}(c(\by)),\quad p_\eps = p_\eps(\by).
\]
Then, using equation \eqref{eq:sphere_vol} we have 
$$\Vol(A_{\eps}) = d \omega_d \int_{\eps \rho}^{\rho} u^{d-1}\param{1-\frac{s(c(\by))}{6d}u^2 + O(u^3)}du\leq \omega_d \rho^d (1- \eps^d) (1+s_{\max}\rho^2), $$
where $s_{\max}= \sup_M (-\frac{s(c(\by))}{6d}) +\delta$, for some $\delta>0$  as in the proof of Proposition \ref{prop:betti_order}. 
Let $\cS$ be a $(\eps\rho/2)$-net of $A_\eps$, i.e. for every $x\in A_{\eps}$ there exists $s\in \cS$ with $\dist(x,s) \le \eps\rho/2$.
Since $M$ is a  compact manifold, we can find a positive  constant $c$ such that $c^{-1}u^d \leq \Vol(B_{u}(p)) \leq cu^d$ for all $p \in M$ and $u \leq r_{\max}$. Therefore,
\eqb \label{eq:points_in_cover}	
\abs{\cS} \leq C \frac{ \vol(A_{\eps}) }{ \inf_{p\in M}{ \Vol(B_{\frac{\eps\rho}{2}}(p) )}} \leq C \frac{\rho^d(1-\eps^d)}{(\eps\rho/2)^d} = C \frac{1- \eps^d}{\eps^d},
\eqe
where the constant $C$ changes in each inequality (see Remark \ref{rem:const}), and depends only on the metric $g$. In other words,  the bound for the number of points in this net only depends on $\eps$ and the metric $g$. If for all $s\in \cS$, we have $\cP_n\cap B_{\rho(1- \eps/2)}(s)\neq \emptyset$, then the triangle inequality implies that $A_{\eps} \subset B_\rho(\cP_n)$.
Notice that since we are conditioning on the event $\set{\cP_n\cap B_{r_2} = \emptyset}$, and using spatial independence  property of the Poisson process, for every $s\in\cS$ we have
\[
 \P_\emptyset\param{\cP_n\cap B_{\rho(1-\eps/2)}(s)= \emptyset  } = e^{-n \vol(B_{\rho(1-\eps/2)}(s)\backslash B_{r_2} )},
\]
and therefore,
\begin{equation}\label{eq:lower_bound_pphi}
\splitb
p_{\eps} & \geq  \P_{\emptyset}\param{ \forall s\in \cS : \cP_n\cap B_{\rho(1-\eps/2)}(s)\ne \emptyset  } \\ 
& =  1- \P_\emptyset\param{ \exists s\in \cS : \cP_n\cap B_{\rho(1-\eps/2)}(s)= \emptyset  } \\ 
& \ge  1- \sum_{s \in \cS} e^{-n \vol(B_{\rho(1-\eps/2)}(s)\backslash B_{r_2})} \\ 
& \geq  1- C \max_{s\in \cS}e^{-n \vol(B_{\rho(1-\eps/2)}(s)\backslash B_{r_2} )},
\splite
\end{equation}
where  for the last inequality we used \eqref{eq:points_in_cover}, and the fact that $\eps$ is fixed.
In order to estimate $\vol(B_{\rho(1-\eps/2)}(s) \backslash B_{r_2})$ we look at the radial arc-length parametrized geodesic $\gamma$, from $c(\by)$ to $s$ (see Figure \ref{fig:geodesic}). This geodesic first enters $B_{\rho(1-\eps/2)}(s) \backslash B_{r_2}$ at a point $p_{\text{in}}= \gamma(r_2)$ which is at distance $r_2$ from $c(\by)$ and leaves it through $p_{\text {out}}= \gamma(\rho(1-\eps/2)+ \dist(s, c(\by)))$. Note that since $s\in A_{\eps}$, we have that $\dist(s,c(\by))\ge \eps\rho$. Therefore,
$$\dist(p_{\text{in}},p_{\text{out}}) = \rho(1-\eps/2)+ \dist(s, c(\by))-r_2 \geq \rho(1+\eps/2) - r_2.$$

%===============================================================================

\begin{figure}[h]
\centering
\includegraphics[scale=0.3]{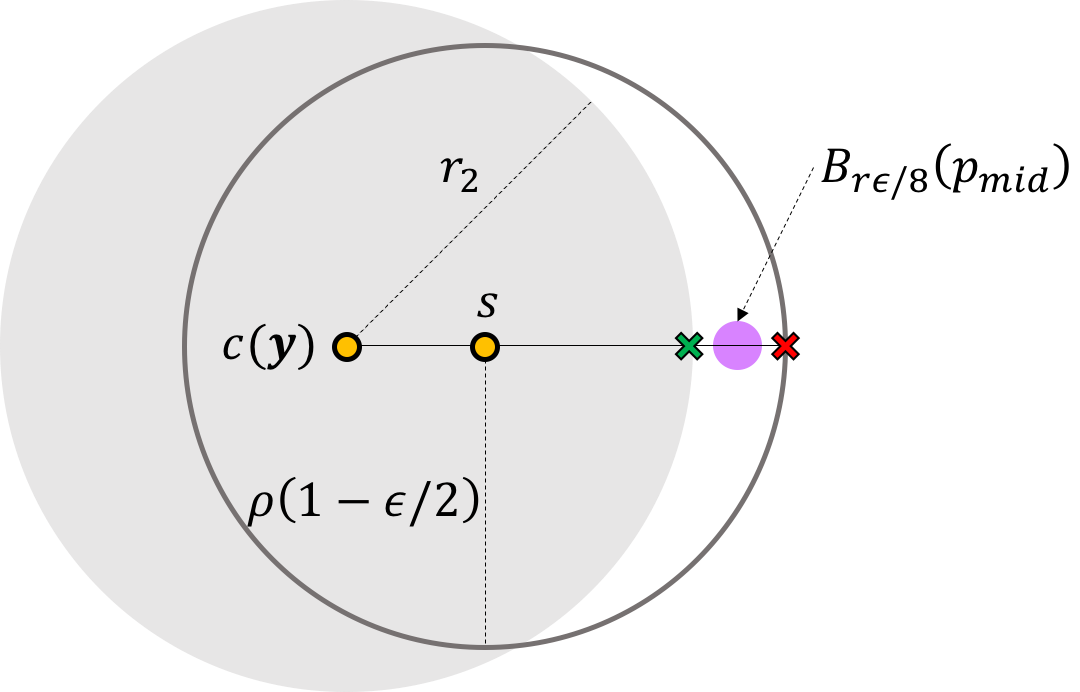}
\caption{\label{fig:geodesic}  Bounding the volume of $R = B_{\rho(1-\eps/2)} \backslash B_{r_2}$.The solid line going through $c(\by)$ and $s$ represents the geodesic. This geodesic enters the region $R$ through the green marker ($p_{\text{in}}$) and leaves through the red marker ($p_{\text{out}}$). The small shaded disc is a ball of radius $r\eps/8$ which we show to be contained inside $R$, and we use its volume as a lower bound for $\vol(R)$. }
\end{figure}

%===============================================================================

Now, take 
\begin{equation}\label{eq:r1r2}
r_2= r(1+  \xi),\quad \text{ and } \quad r_1=r (1-\xi^2/2c_g^2) 
\end{equation}
 with $\xi =o(\eps)$ ($\xi$ is to be determined later). Then, these satisfy the conditions in the statement of Lemma \ref{lem:Betas} and
$$\dist(p_{\text{in}}, p_{\text{out}}) \geq r_1(1+\eps/2) - r_2 >{ \frac{r \eps}{4},}$$ 
for $\xi =o( \epsilon)$. As a consequence, the ball of radius $\frac{r \eps}{8}$ centered at the midpoint $p_{\text{mid}}$ from $p_{\text{in}}$ to $p_{\text{out}}$ is  completely contained in $B_{\rho(1-\eps/2)}(s) \backslash B_{r_2}$, i.e.
$$B_{r \eps/8}(p_{\text{mid}}) \subset B_{\rho(1-\eps/2)}(s) \backslash B_{r_2}.$$
Hence $\vol(B_{\rho(1-\eps/2)}(s) \backslash B_{r_2}) \geq C \eps^d r^d$, for some constant $C>0$, which  depends on the metric, but can be taken to be independent of $\by$ and $s$. Putting this back into \eqref{eq:lower_bound_pphi}, we have
$$1-Ce^{-Cn\eps^dr^d} \leq p_{\eps}(\by) \leq 1,\quad \forall \by \in M^{k+1}.$$
Since $\Lambda  \to \infty$ we conclude that  $p_\eps(\by)$ converges uniformly to $1$, as $\eps> \eps_0>0$, for some $\eps_0$, i.e. stays bounded away from zero.

%===============================================================================

\noindent\textbf{Step 2:} We estimate the integral in  \eqref{eq:mean_C_k_th} and show that $\mean{\beta_k^\eps(r)} = \Omega( n \Lambda^{k-2} e^{- \Lambda })$.

Combining the result of step $1$ with \eqref{eq:mean_C_k_th}, we have that for sufficiently large $n$
\[
\meanx{\beta_k^\eps(r)} \geq \frac{1}{2} \frac{n^{k+1}}{(k+1)!} \int_{M^{k+1}} h_r^{\eps}(\by)  e^{-n \Vol (B_{r_2}(c(\by))) } \abs{ \dvol_g(\by) },
\]
Next, we use normal coordinates to estimate the volume of $B_{\rho(\by)}(c(\by))$ for $\rho(\by) < r$, as in Corollary \ref{cor:Volume}. This use is very similar to that used in equation \eqref{eq:volume_estimate_normal_coords}, in the proof of Lemma \ref{lem:crit_pts}, yielding
$$e^{-n\vol(B_{r_2}(c(\by)))} \geq e^{-\Lambda_{r_2} }(1+s_{\min}\Lambda_{r_2} r^2),$$
where $\Lambda_{r_2}=\omega_d nr_2^d$ and $s_{\min}= \inf_{c\in M} \frac{s(c)}{6(d+2)} +\delta$, for some $\delta>0$. We proceed as in the proof of Proposition \ref{prop:betti_order} and then
\[
\meanx{\beta_k^\eps(r)} \geq   D^\eps_k (1+s_{\min}\Lambda r^2)(1-c_{_R} r^2) n \Lambda^k e^{-\Lambda_{r_2}} \int_{\frac{r_1}{r}}^{1}  s^{dk-1} ds,
\]
where $c_{_R}=\sup_{V \in Gr(k,TM)} \abs{-\frac{Ric^{V}}{3}}+ \nu$ and
\begin{eqnarray}\nonumber
 D_k^\eps & := & \frac{1}{2\omega_d^k(k+1)!} \int_M \abs{ \dvol_g(c) } \int_{Gr(k, T_cM)}  {\dmu_{k,d} }(V) \times \\ \nonumber
 & &  \times \abs{  \prod_{i=1}^{k+1} \int_{\mathbb{S}_1(V)}   \inf_{s \in (r_1, r)}\sqrt{\abs{ \det(g(\exp_c(u w_i))) }} \  \dvol_{\mathbb{S}_1(V)}(w_i) }  \Upsilon_1^{d-k}(\textbf{w}) h^\eps(\exp_c( u \textbf{w}))
\end{eqnarray}
with $h^\eps(\by) := h(\by)\ind\set{\phi(\by) \ge \eps}$. First we notice that from \eqref{eq:RiemannianMeasure} we have that if $r$ and $r_1$ are small enough, then each of the  terms $\det(g(\exp_c(u w_i)))$ is as close to $1$ as we want. Secondly, we recall that
\[
	\phi(\exp_c(u \textbf{w})) = \frac{1}{2} \min_{w \in \partial\Delta(\exp_c( \textbf{w} ))} \abs{ w },
\]
which is a nonnegative continuous function of $\textbf{w}$ vanishing along a measure zero set consisting of those $\textbf{w}$ for which $0 \in T_cM$ is contained in a face of their convex hull. This is a measure zero set and  for any sufficiently small $\eps>0$, the support of $\ind \set{\phi(\exp_c(u \textbf{w})) \ge \eps}$  has a nonzero measure. Also, notice that if $\phi(\exp_c(u \textbf{w})) \ge \eps$  then $\Upsilon_1(\textbf{w})$ (the $k$-volume of the simplex $\Delta(\textbf{w})$) is  bounded from below by a quantity of the order of $\epsilon^k$. Putting these two facts together we conclude that for all sufficiently small $\epsilon>0$, we do have $ D_k^\eps > 0$. Moreover, since we set $r_1=r(1-\xi^2/2c_g^2)$ we have $\int_{\frac{r_1}{r}}^{1}s^{dk-1} ds = \xi^2/2c_g^2+O((\xi^2/2c_g^2)^2 ) \ge \xi^2/3c_g^2$, and thus
\begin{equation}\label{eq:mean_theta_bound}
	\mean{\beta_k^\eps(r)} \geq  \frac{D_k^\eps}{3c_g^2} (1+s_{\min}\Lambda r^2)(1-c_{_R} r^2)^{k+1}  n \Lambda^{k} \xi^2  e^{- \Lambda_{r_2}} .
\end{equation}
Note that in the exponent we have $\Lambda_{r_2}=\Lambda (1+\xi)^d$ (since we set $r_2 = r(1+\xi)$), and therefore $e^{-\Lambda_{r_2}}=e^{-\Lambda (1+\xi)^d}=e^{-\Lambda} e^{\Lambda (-d\xi + o(\xi))}$. However, as $\Lambda \to + \infty$ this second exponential, $e^{\Lambda (-d\xi + o(\xi))}$, is bounded from below if and only if $\xi$ decays at least as fast as $O(\Lambda^{-1})$. Therefore, we take $\xi=\Lambda^{-1}$, in which case $e^{-\Lambda_{r_2}} \geq \frac{1}{2}e^{-d} e^{-\Lambda}$, and putting it back into \eqref{eq:mean_theta_bound} we have
$$\mean{\beta_k^\eps(r)} \geq  \frac{D_k^\eps}{6c_g^2} e^{-d}(1+s_{\min}\Lambda r^2)(1-c_{_R} r^2)^{k+1}  n \Lambda^{k-2}   e^{- \Lambda}.$$
Now if we take $\Lambda \to + \infty$, but $\Lambda r^2 \to 0$, then the statement follows.\\
\end{proof}

%===============================================================================

\begin{rem}\label{rem:Scalar_Curvature_phi}
A more refined conclusion of the statement above would be that given a sufficiently small $\eps>0$ and $r>0$, then for $r_1=r(1-\frac{\Lambda^{-2}}{2c_g^2})$ and $r_2=r(1+\Lambda^{-1})$ the following holds. 
\begin{enumerate}
\item If $\Lambda \rightarrow + \infty$ and $\Lambda r^2 \to 0$, then $\mean{\beta_k^\eps(r)} \geq C(k,d,\eps,g) n \Lambda^{k-2} e^{-\Lambda}$, for some $C(k,d,\eps,g)>0$ depending on $k,d,\eps$ and the metric $g$.
\item If $(M,g)$ has everywhere positive scalar curvature, then as $\Lambda \to + \infty$ one does not need to assume that $\Lambda r^2 \to 0$. Instead it is enough to assume this stays bounded, in which case
$$\mean{\beta_k^\eps(r)} \geq C n \Lambda^{k-2}(1+ s_{\max}\Lambda r^2) e^{-\Lambda},$$
for some $C>0$ as in the previous bullet.
\end{enumerate}
\end{rem}

%===============================================================================

Putting all the previous lemmas together, we can now prove the main result of this section.

%===============================================================================

\begin{proof}[Proof of Proposition \ref{prop:betti_order2}]

It follows from Lemma \ref{lem:Betas} that $\mean{\beta_k(r)} \geq \mean{\beta_k^{\eps}(r)}$, and using  Lemma \ref{lem:theta_cycle_order}, we have
$$ \mean{\beta_k(r)} \geq a_{k} n \Lambda^{k-2} e^{-\Lambda}, $$
for some $a_{k}$ depending only on $k$, $g$ and $d$.\\
\end{proof}

%===============================================================================

\section{Second Moment Calculations}\label{sec:second_moment}

%===============================================================================

The following proposition will be used to provide the lower threshold in Theorem \ref{thm:main}.
It uses a second moment argument, based on  Chebyshev's inequality $\prob{\abs{X-\mean{X}}\ge a} \le \frac{\var{X}}{a^2}$.

%===============================================================================

\begin{prop}\label{prop:second_moment}
For all sufficiently small $\epsilon>0$, and $1 \leq k \leq d-1$ the following holds.
Fix $\delta \in (0,1)$. If $\Lambda = \log n + (k-2)\log\log n -w(n)$, with $w(n)\to\infty$, then
$$\lim_{n \rightarrow + \infty} \prob{ \beta_k^{\epsilon} (r) >\delta\ \mean{\beta_k^{\eps}(r)} } =1 . $$
\end{prop}

%===============================================================================

\begin{proof}(Proof of proposition \ref{prop:second_moment})
Since $\beta_k^\eps(r) \ge 0$, then using Chebyshev's inequality we have
\[
\prob{\beta_k^\eps(r) \le \delta\ \mean{\beta_k^{\eps}(r)}} \le \frac{\Var(\beta_k^{\epsilon} (r))}{(1-\delta)^2\meanx{\beta_k^{\epsilon}(r)}^2},
\]
where  $\Var(\beta_k^{\epsilon} (r))= \meanx{ (\beta_k^{\epsilon}(r))^2 }- \meanx{\beta_k^{\epsilon}(r)}^2$. 
Thus, showing that the right hand side converges to zero will prove the statement.
Recall from equation \eqref{eq:Beta_k_Epsilon} that $\beta_k^\eps(r) = \sum_{\cY\subset \cP_n} g_r^\eps(\cY,\cP_n)$ and so  we can write 
\begin{eqnarray}
(\beta_k^{\epsilon}(r))^2 & = &  \sum_{\cY_1 ,\cY_2 \subset \cP_n } g_r^\eps(\cY_1,\cP_n) g_r^\eps(\cY_2,\cP_n) ,
\end{eqnarray}
where $\cY_i$ ($i=1,2$) run through all subsets of $\cP_n$ with $(k+1)$-points. Next, 
defining
\[
\Phi_r(\cY_1,\cY_2) := \ind\set{B_r(c(\cY_1))\cap B_r(c(\cY_2)) = \emptyset}.
\]
we can write
\begin{eqnarray}\nonumber
\meanx{ (\beta_k^{\epsilon} (r))^2} & = &    \meanx{ \sum_{\cY_1, \cY_2 \subset \cP_n}    g_r^\eps(\cY_1,\cP_n) g_r^\eps(\cY_2,\cP_n) } \\ \nonumber
& = &   \meanx{  \sum_{\cY_1, \cY_2 \subset \cP_n}   g_r^\eps(\cY_1,\cP_n) g_r^\eps(\cY_2,\cP_n)  \Phi_{2r}(\cY_1,\cY_2) }\\ \nonumber
& +&    \meanx{ \sum_{\cY_1, \cY_2 \subset \cP_n}   g_r^\eps(\cY_1,\cP_n) g_r^\eps(\cY_2,\cP_n) (1-\Phi_{2r}(\cY_1,\cY_2)) } \\ \nonumber
& = & \meanx{T_1} +  \meanx{T_2},
\end{eqnarray}
where  $T_1$ and $T_2$ are the first and second sums appearing in the expectation above.  Thus, we have
\begin{eqnarray}\nonumber
\Var(\beta_k^{\epsilon}(r)) & = & \meanx{(\beta_k^{\epsilon}(r))^2}- \meanx{\beta_k^{\epsilon}(r)}^2 \\ \nonumber
& = & ( \meanx{T_1}-\meanx{\beta_k^{\epsilon}(r)}^2 ) +   \meanx{T_2},
\end{eqnarray}
and our next step is to bound the terms $( \meanx{T_1}-\meanx{\beta_k^{\epsilon}}^2 )$ and $\meanx{T_2}$ separately.
Using Palm theory (Theorem \ref{thm:prelim:palm}) we can write
\begin{equation}\label{eq:second_moment_1}
\begin{split}
\meanx{\beta_k^{\epsilon}(r)}^2 & =  \frac{n^{2k+2}}{((k+1)!)^2} \meanx{g_r^{\epsilon}(\cY'_1 , \cY'_1 \cup \cP_n)g_r^{\epsilon}(\cY'_2 , \cY'_2 \cup \cP'_n)} \\ 
\meanx{T_1 } & =  \frac{n^{2k+2}}{((k+1)!)^2} \meanx{g_r^{\epsilon}(\cY'_1 , \cY' \cup \cP_n)g_r^{\epsilon}(\cY'_2 , \cY' \cup \cP_n)  \Phi_{2r}(\cY_1,\cY_2)   },
\end{split}
\end{equation}
where $\cY'_1, \cY'_2$ are independent sets of $k+1$ points uniformly distributed in $M$, $\cY'= \cY'_1 \cup \cY'_2$, and $\cP_n'$ an independent copy of $\cP_n$. Thus, using \eqref{eq:second_moment_1} we have 
{
 \begin{eqnarray}\nonumber
 \meanx{T_1}-\meanx{\beta_k^{\epsilon}(r)}^2  & =  &  \frac{n^{2k+2}}{((k+1)!)^2}  \Big(  \meanx{g_r^{\epsilon}(\cY'_1 , \cY' \cup \cP_n)g_r^{\epsilon}(\cY'_2 , \cY' \cup \cP_n) \Phi_{2r}(\cY'_1,\cY'_2) } \\ \nonumber
& & \ \ \ \ - \meanx{  g_r^{\epsilon}(\cY'_1 , \cY'_1 \cup \cP_n)g_r^{\epsilon}(\cY'_2 , \cY'_2 \cup \cP'_n)  \Phi_{2r}(\cY'_1,\cY'_2)}    \\ \nonumber
& & \ \ \ \ - \meanx{   g_r^{\epsilon}(\cY'_1 , \cY'_1 \cup \cP_n)g_r^{\epsilon}(\cY'_2 , \cY'_2 \cup \cP'_n) (1- \Phi_{2r}(\cY'_1,\cY'_2))   }   \Big) \\ \nonumber
& \leq &   \frac{n^{2k+2}}{((k+1)!)^2}  \Big(  \meanx{g_r^\eps(\cY'_1 , \cY'_1 \cup \cP_n)g_r^\eps(\cY'_2 , \cY'_2 \cup \cP_n)  \Phi_{2r}(\cY'_1,\cY'_2)}    \\ \nonumber
& & \ \ \ \ - \meanx{  g_r^\eps(\cY'_1 , \cY'_1 \cup \cP_n)g_r^\eps(\cY'_2 , \cY'_2 \cup \cP'_n)  \Phi_{2r}(\cY'_1,\cY'_2)  }  \Big). \\ \nonumber
& = & \frac{n^{2k+2}}{((k+1)!)^2} \meanx{\Delta g_r^\eps}
\end{eqnarray}}
where
$$\Delta g^\eps_r:= \param{g_r^\eps(\cY'_1 , \cY'_1 \cup \cP_n)g_r^\eps(\cY'_2 , \cY'_2 \cup \cP_n) - g_r^\eps(\cY'_1 , \cY'_1 \cup \cP_n)g_r^\eps(\cY'_2 , \cY'_2 \cup \cP'_n)}\Phi_{2r}(\cY'_1,\cY'_2).$$
Similarly to \cite{bobrowski_vanishing_2015}, we can show that $\mean{\Delta g_r^\eps} = 0$ as follows.
Consider the conditional distribution of $\Delta g_r^\eps$ given $\cY'_1,\cY'_2$, and denote $\E_{\cY'_1,\cY'_2}\set{\cdot} = \cmean{\cdot}{\cY'_1,\cY'_2}$. If $\Delta g_r^\eps \ne 0$ then necessarily $B_{2r}(c(\cY'_1))\cap B_{2r}(c(\cY'_2))=\emptyset$. Using the spatial independence property of the Poisson process, together with the fact that the value of $g_r^\eps(\cY'_i, \cY'_i\cup \cP_n)$ only depends on the points of $\cP_n$ lying inside $B_{2r}(c(\cY'_i))$,
we conclude that
\[
\splitb
\E_{\cY'_1,\cY'_2}\set{g_r^\eps(\cY'_1,\cY'_1\cup\cP_n)g_r^\eps(\cY'_2,\cY'_2\cup\cP_n)}   &= \E_{\cY'_1,\cY'_2}\set{g_r^\eps(\cY'_1,\cY'_1\cup\cP_n)}\E_{\cY'_1,\cY'_2}\set{g_r^\eps(\cY'_2,\cY'_2\cup \cP_n)} \\
&= \E_{\cY'_1,\cY'_2}\set{g_r^\eps(\cY'_1,\cY'_1\cup\cP_n)}\E_{\cY'_1,\cY'_2}\set{g_r^\eps(\cY'_2,\cY'_2\cup \cP'_n)} \\
\splite
\]
since $\cP_n$ and $\cP_n'$ are independent and have the same distribution. Thus, we have that $\E_{\cY_1,\cY_2}\set{\Delta g_r} = 0$.
Consequently, $\mean{\Delta g_r} = \mean{\E_{\cY_1,\cY_2}\set{\Delta g_r} }=0$.

Next, we wish to bound  $\meanx{T_2}$. We start by writing,
\begin{eqnarray}\nonumber
T_2 & = &   \sum_{\cY_1, \cY_2 \subset \cP_n}   g_r^\eps(\cY_1,\cP_n) g_r^\eps(\cY_2,\cP_n) (1-\Phi_{2r}(\cY_1,\cY_2))\\ \label{eq:I_j_First_Appearence}
& = & \sum_{j=0}^{k+1} \sum_{ \abs{ \cY_1 \cap \cY_2 } =j} g_r^\eps(\cY_1,\cP_n) g_r^\eps(\cY_2,\cP_n)  (1-\Phi_{2r}(\cY_1,\cY_2)).
\end{eqnarray}
In the following, it will be convenient to refer to the inner sum as $I_j$. Lemmas \ref{lem:I_0_Computations} and \ref{lem:I_j_Computations} respectively provide upper bounds on $\mean{I_0}$ and  $\mean{I_j}$ for $j \geq 1$. The largest of these is the upper bound for $\mean{I_0}$ which yields
$$\meanx{T_2} \leq C n \Lambda^{2k+1} \Lambda^{-4} e^{-\Lambda} \left( e^{- \omega_d \alpha \Lambda / \omega_{d-1} } + \alpha^d \right), $$ for any $\alpha < 1$.
Using the fact that $\meanx{\beta_k^{\epsilon}(r)}\ge a_k n \Lambda^{k-2} e^{- \Lambda}$ and $\Lambda r \to 0$, we get
\begin{eqnarray}\nonumber
\frac{\meanx{T_2}}{\meanx{\beta_k^{\epsilon}(r)}^2} & \leq & C \frac{ n \Lambda^{2k+1} \Lambda^{-4} e^{-\Lambda}  }{ n^2 \Lambda^{2k-4} e^{- 2\Lambda} } ( e^{- \omega_d \alpha \Lambda / \omega_{d-1} } + \alpha^d ) \leq \frac{\Lambda e^{\Lambda}}{n} (  e^{- \omega_d \alpha \Lambda / \omega_{d-1}}+\alpha^d  ).
\end{eqnarray}
Taking $\Lambda=\log n + (k-2) \log \log n - w(n)$ and $\alpha = \frac{k \omega_{d-1}}{\omega_d} \frac{\log \log n}{\log n}$ we have
\begin{eqnarray}\nonumber
\frac{\meanx{T_2}}{\meanx{\beta_k^{\epsilon}(r)}^2} & \leq & C e^{-w(n)} (\log n)^{k-1} \left( e^{- \frac{\omega_d}{ \omega_{d-1}} \alpha \log n} + \alpha^d \right)  \\ \nonumber
& \leq & C e^{-w(n)} (\log n)^{k-1} \left( \frac{1}{(\log n)^k} + \frac{(\log \log n)^d}{(\log n)^d} \right)  \\ \nonumber
& \leq & C e^{-w(n)}  \left( \frac{1}{\log n} + \frac{(\log \log n)^d}{(\log n)^{d+1-k}} \right) \to 0, \ \text{as} \ n \to + \infty .
\end{eqnarray}
\end{proof}

%===============================================================================

To conclude this section, we are left with bounding the $\mean{I_j}$ terms used in the proof of the previous proposition. Recall that these are defined by
$$I_j=\sum_{ \abs{ \cY_1 \cap \cY_2 } =j} g_r^\eps(\cY_1,\cP_n) g_r^\eps(\cY_2,\cP_n)  (1-\Phi_{2r}(\cY_1,\cY_2)),$$
and first appeared in \eqref{eq:I_j_First_Appearence}. The following two statements provide the needed bounds.

%===============================================================================

\begin{lem}[Estimate on $I_0$]\label{lem:I_0_Computations} For any $0 < \alpha < 1$ we have
$$\meanx{I_0} \leq C n \Lambda^{2k+1} \Lambda^{-4} e^{-\Lambda} \left( e^{- \omega_d \alpha \Lambda / \omega_{d-1} } + \alpha^d \right)  . $$
\end{lem}

%===============================================================================

\begin{lem}[Estimates on $I_j$, for $j \geq 1$]\label{lem:I_j_Computations} For any $0 < \alpha < 1$, we have
$$\meanx{I_j} \leq C n \Lambda^{2k+1-j} r^{j(k-j)+1} \Lambda^{-4} e^{-\Lambda} \left( e^{- \omega_d \alpha \Lambda / \omega_{d-1} } + \alpha^{d-j+1} \right)  . $$
\end{lem}

%===============================================================================

\begin{proof}[Proof of Lemma \ref{lem:I_0_Computations}]
We start estimating $\mean{I_0}$ by using Palm theory (Corollary \ref{cor:prelim:palm2}) as follows.
\begin{equation}\label{eq:Integral_I_0}
\begin{split}
\meanx{I_0} &=   \meanx{ \sum_{\abs{ \cY_1 \cap \cY_2 } = 0}   g_r^\eps(\cY_1,\cP_n) g_r^\eps(\cY_2,\cP_n)  (1-\Phi_{2r}(\cY_1,\cY_2))}  \\ 
&=    \frac{n^{2k+2}}{ ((k+1)!)^2} \meanx{  g_r^\eps(\cY'_1, \cY'\cup \cP_n) g_r^\eps(\cY'_2, \cY'\cup \cP_n)  (1-\Phi_{2r}(\cY'_1,\cY'_2))} \\ 
& \leq  \frac{n^{2k+2}}{ ((k+1)!)^2} \int_{M^{2k+2}} h_r^\eps( \by_1) h_r^\eps( \by_2) e^{-n \Vol(\by_1, \by_2)} (1-\Phi_{2r}(\by_1,\by_2))\abs{ \dvol_g(\by_1,\by_2) }.
\end{split}
\end{equation}
where $\by_1 , \by_2 \in M^{k+1}$ are $(k+1)$-tuples of points, 
$\Vol(\by_1, \by_2) = \vol(B_{r_2}(c_1) \cup B_{r_2}(c_2))$,
and $c_i = c(\by_i)$.

Next, we need to compare the volume of this union with the Euclidean volume of Euclidean balls with a slightly different radius. Using Lemma \ref{lem:RiemannianBallsComparison} to compare the balls, and Lemma \ref{lem:RiemannianMeasureComparison} to compare the volumes yields
\begin{eqnarray}\nonumber
\Vol(B_{r_2}(c_1) \cup B_{r_2}(c_2)) & \geq & \Vol(B^E_{(1-\nu' r)r_2}(c_1) \cup B^E_{(1-\nu' r)r_2}(c_2)) \\ \nonumber
& \geq & (1-\nu r^2) \Vol^E(B^E_{(1-\nu' r)r_2}(c_1) \cup B^E_{(1-\nu' r)r_2}(c_2)) \\ \nonumber
& \geq & (1-\nu r^2) \Vol^E(B^E_{(1-\nu' r)r}(c_1) \cup B^E_{(1-\nu' r)r}(c_2)) \\ \nonumber
& \geq & (1-\nu r^2) (2 \omega_d (1-\nu' r)^d r^d - \Vol^E(B^E_{(1-\nu' r)r}(c_1) \cap B^E_{(1-\nu' r)r}(c_2) ) \\ \nonumber
& \geq & (1- (d\nu' r + \nu r^2) +o(r^2))  \omega_d r^d  \\ \nonumber
& &\times \left( 1 +\frac{\omega_{d-1}}{\omega_d} \frac{\dist(c_1,c_2)}{r} + O \left( \frac{\dist(c_1,c_2)^2}{r} \right)  \right).
\end{eqnarray}
where we used the fact that $r_2 > r$, and in the last inequality we used the following Tailor expansion for the volume of the intersection of two balls (see Appendix C in \cite{bobrowski_vanishing_2015}), 
\begin{equation}\label{eq:vol_taylor}
\Vol^E(B^E_r(c_1) \cap B^E_r(c_2)) = \omega_d r^d - \omega_{d-1} r^{d-1} \dist(c_1,c_2)+O(r^{d-2} \dist^2(c_1,c_2)).
\end{equation}
 Next, let $\alpha>0$, to be determined later, and separate the integration in \eqref{eq:Integral_I_0} into two regions
\begin{eqnarray}\nonumber
\Omega_1 & = & \lbrace  (\by_1, \by_2) \in M^{2k+2} \ \vert \ \frac{\dist(c_1,c_2)}{r} \leq \alpha  \rbrace, \\ \nonumber
\Omega_2& = & \lbrace  (\by_1, \by_2) \in M^{2k+2} \ \vert \ \frac{\dist(c_1,c_2)}{r}> \alpha  \rbrace.
\end{eqnarray}
Then, in $S_1$ we have $\Vol(\by_1, \by_2) \geq  (1- (d\nu' r + \nu r^2) ) \omega_d r^d$. Moreover, using a similar change of variables to that of Lemma \ref{lem:Integral}, and taking $c_2$ to be in polar coordinates around $c_1$, we have

\[
\begin{split}
I_0^{(1)} &:=\int_{\Omega_1}  h_r^{\epsilon} ( \by_1)  h_r^{\epsilon}( \by_2)  e^{-n \Vol(\by_1, \by_2)} (1-\Phi_{2r}(\by_1,\by_2))\abs{ \dvol_g(\by_1,\by_2) }  \\
\leq &  \int_{\Omega_1} h_r^{\epsilon}( \by_1) h_r^{\epsilon}( \by_2) e^{-(1- (d\nu' r + \nu r^2) ) \Lambda} \abs{ \dvol_g(\by_1,\by_2) } \\
 \leq & C e^{-(1- (d\nu' r + \nu r^2) ) \Lambda} \int_M \abs{ \dvol_g(c_1) } \int_0^{\alpha r} ds \int_{\mathbb{S}_1(T_{c_1} M)} s^{d-1} \dvol_{\mathbb{S}_1(T_{c_1} M)}(w)  \\ 
  \times &\prod_{i=1}^2 \int_{r_1}^r d u_i  \int_{Gr(k,d)} u_i^{k(d-k)} d\mu_{k,d}(V) \int_{(\mathbb{S}_1(V))^{k+1}} u_i^{(k-1)(k+1)} \dvol_{(\mathbb{S}_1(V))^{k+1}} (\bw_i)  h^\eps(\exp_{c_i}(u_i\bw_i) ),
\end{split}
\]
with $c_2= \exp_{c_1}(sw)$, $\by_i = \exp_{c_i}(\bw_i)$ and $\mathbb{S}_1(V)$ denotes the unit sphere in the $k$-dimensional vector space $V$. Note that the term $C$ includes not only an upper bound for  metric-related terms, but also a bound for the $\Upsilon$ terms representing the parallelogram-volume. After carrying out the integration in the radial coordinates we get
\begin{eqnarray}\nonumber
I_0^{(1)}& \leq & C e^{-(1- (d\nu' r + \nu r^2) ) \Lambda} (\alpha r )^d r^{2dk} \left( r^{dk} -  r_1^{dk} \right)^2 \\ \nonumber
& \leq &  C e^{-(1- (d\nu' r + \nu r^2) ) \Lambda} (\alpha r )^d r^{2dk} \left( 1- \left( \frac{r_1}{r} \right)^{dk} \right)^2 ,
\end{eqnarray}
for some new constant $C$ depending on $k,d$ and the metric $g$ (see Remark \ref{rem:const}). Then, taking $r_1=(1-\xi^2/2c_g^2)r$ with $\xi=\Lambda^{-1}$ (as in \eqref{eq:r1r2}) we conclude that
\begin{eqnarray}\nonumber
I_0^{(1)}& \leq &  C  e^{-(1- (d\nu' r + \nu r^2) ) \Lambda} r^{d(2k+1)}  \alpha^d \Lambda^{-4}.
\end{eqnarray}  At this stage we require that $\Lambda r \to 0$ as $n \to + \infty$. Then, $e^{-\Lambda + (d\nu' r + \nu r^2) \Lambda} \sim e^{-\Lambda}$, this yields
\begin{eqnarray}\nonumber
I_0^{(1)} & \leq &  C \alpha^d \Lambda^{-4} e^{-\Lambda}  r^{d(2k+1)} .
\end{eqnarray}
We now turn to  evaluate  the integral over $\Omega_2$. Firstly, notice that $\Vol(\by_1, \by_2)$ is increasing with $\dist(c_1,c_2)$ and so attains its minimum value when $\dist(c_1,c_2)=\alpha r$ in this set. Secondly, notice that for the term $(1-\Phi_{2r}(\by_1,\by_2))$ to be nonzero we must have $\dist(c_1,c_2)\le 4r$. Inserting $\dist(c_1,c_2)= \alpha r$ in the Taylor expansion \eqref{eq:vol_taylor} and using  a similar change of variables yields
\[
\begin{split}
I_0^{(2)} &:=\int_{\Omega_2}  h_r^{\epsilon} ( \by_1)  h_r^{\epsilon}( \by_2)  e^{-n \Vol(\by_1, \by_2)}(1-\Phi_{2r}(\by_1,\by_2)) \abs{ \dvol_g(\by_1,\by_2) } \\ 
& \leq   \int_{\Omega_2} h_r^{\epsilon}( \by_1) h_r^{\epsilon}( \by_2) e^{-(1- (d\nu' r + \nu r^2) ) (1+ \omega_d \alpha / \omega_{d-1} ) \Lambda} \abs{ \dvol_g(\by_1,\by_2) } \\ 
& \leq  {C e^{- (1+ \omega_d \alpha / \omega_{d-1} ) \Lambda} }\int_M \abs{ \dvol_g(c_1) } \int_{\alpha r}^{4r} ds \int_{\mathbb{S}_1(T_{c_1}M)} s^{d-1} \dvol_{\mathbb{S}_1(T_{c_1}M)}(w) \times \\ 
&  \times \prod_{i=1}^2 \int_{r_1}^r d u_i  \int_{Gr(k,d)} u_i^{k(d-k)} d\mu_{k,d} \int_{(\mathbb{S}_1(V))^{k+1}} u_i^{(k-1)(k+1)} \dvol_{(\mathbb{S}_1(V))^{k+1}} (\bv_i) h_r^{\epsilon}( \exp_{c_i}(u_i\bw_i)) .
\end{split}
\]
In this case the integration in the radial coordinates yields
\begin{eqnarray}\nonumber
\int_{\Omega_2}  h_r^{\epsilon} ( \by_1) h_r^{\epsilon}( \by_2)  e^{-n \Vol(\by_1, \by_2)} \abs{ \dvol_g(\by_1,\by_2) }  \leq  C \Lambda^{-4} e^{-\Lambda} e^{- \omega_d \alpha \Lambda / \omega_{d-1} }  r^{d(2k+1)}.
\end{eqnarray}
Putting these all together into \eqref{eq:Integral_I_0} we have
$$\meanx{I_0} \leq C n \Lambda^{2k+1} \Lambda^{-4} e^{-\Lambda} \left( e^{- \omega_d \alpha \Lambda / \omega_{d-1} } + \alpha^d \right)  . $$
\end{proof}

%===============================================================================

\begin{proof}[Proof of Lemma \ref{lem:I_j_Computations}]
As before we evaluate $I_j$ using Palm theory (Corolllary \ref{cor:prelim:palm2})
\begin{equation}\label{eq:Integral_I_j}
\begin{split}
\meanx{I_j} &=   \meanx{ \sum_{\abs{ \cY_1 \cap \cY_2 } = j}   g_r^\eps(\cY_1,\cP_n) g_r^\eps(\cY_2,\cP_n)  (1-\Phi_{2r}(\cY_1,\cY_2))}  \\ 
&=    \frac{n^{2k+2-j}}{ j!((k+1-j)!)^2} \meanx{  g_r^\eps(\cY'_1, \cY'\cup \cP_n) g_r^\eps(\cY'_2, \cY' \cup \cP_n)  (1-\Phi_{2r}(\cY'_1,\cY'_2))} \\ 
& \leq  \frac{n^{2k+2-j}}{ j!((k+1-j)!)^2} \int_{M^{2k+2}} h_r^\eps( \by_1) h_r^\eps( \by_2) e^{-n \Vol(\by_1, \by_2)} (1-\Phi_{2r}(\by_1,\by_2))\dvol_g(\by).
\end{split}
\end{equation}
where 
\[
\begin{split}
\by &= (y_1,\ldots, y_{2k+2-j})\subset M^{2k+2-j},\\
\by_1 &= (y_1,\ldots, y_{k+1}),\\
\by_2 &= (y_1,\ldots, y_j, y_{k+2},\ldots, y_{2k+2-j}).
\end{split}
\]
We can repeat the previous computations to show that
\[
\begin{split}
\Vol(\by_1,\by_2) & \ge (1- (d\nu' r + \nu r^2))  \omega_d r^d  \\
& \times \left( 1 +\frac{\omega_{d-1}}{\omega_d} \frac{\dist(c_1,c_2)}{r} + O \left( \frac{\dist(c_1,c_2)^2}{r} \right)  \right),
\end{split}
\]
where again $c_i = c(\by_i)$.
Proceeding using the same steps as in the previous proof, we fix $\alpha>0$ and separate the integration in \eqref{eq:Integral_I_j} into two regions
\begin{eqnarray}\nonumber
\Omega_1 & = & \lbrace  \by \in M^{2k+2-j} \ \vert \ \frac{\dist(c_1,c_2)}{r} \leq \alpha  \rbrace, \\ \nonumber
\Omega_2& = & \lbrace  \by \in M^{2k+2-j} \ \vert \ \frac{\dist(c_1,c_2)}{r}> \alpha  \rbrace.
\end{eqnarray}

To proceed, we need a slightly more involved version of the change of variables used than the one we used in Lemma \ref{lem:I_0_Computations} for $I_0$. The iterated integration goes as follows.

\begin{itemize}
\item Integrate over $M$ to determine the first center $c_1$.
\item Find $\by_1$ on  a $(k-1)$-sphere centered at $c_1$ with radius in $(r_1 , r)$. This goes very much along the same lines as Lemma \ref{lem:Integral}
\item 
Find the second center $c_2$ so that the points $c_2 \in E  = E(y_1 , \ldots , y_j)$ (i.e.~$y_1,\ldots y_j$ are equidistant to $c_2$). 
Since $c_1\in E$ as well,  in order to find $c_2$ we integrate in $E$ using geodesic polar coordinates around $c_1$. Note that $E$  is of dimension $d-j+1$.
\item 
Fixing $c_2$, we need to choose the remaining $k+1-j$ points $(y_{k+1},\ldots, y_{2k+1-j})$. These points, together with  $(y_1,\ldots,y_j)$  span a $k$-dimensional vector space $V_2 \subset T_{c_2}M$. However, we must integrate over those spaces $V_2$ that contain the $j$ -dimensional vector space generated by the $(y_1, \ldots , y_j)$, i.e. the span of the $\langle ( \nabla \rho^2_{y_i} )_{c_2} \rangle_{i=1}^j$. Hence, we integrate  over the subspaces  $W$ that satisfy $V_2=\langle ( \nabla \rho^2_{p_i} )_{c_2} \rangle_{i=1}^j \oplus W$. Note that these are $(k-j)$-dimensional subspaces of $T_{c_2}M$.
\item Finally, we determine the remaining $k+1-j$ points, in geodesic polar coordinates, by integrating over a sphere in the vector subspace $V_2 \subset T_{c_2}M$.
\end{itemize}
In order to make the notation lighter we shall omit the reference to the points over which we are integrating in many occasions. Recalling that in $S_1$ we have $\Vol(\by_1, \by_2) \geq  (1- (d\nu' r + \nu r^2) ) \omega_d r^d$, we have
\[
\begin{split}
I_j^{(1)} &:=\int_{\Omega_1}  h_r^{\epsilon} ( \by_1)  h_r^{\epsilon}( \by_2)  e^{-n \Vol(\by_1, \by_2)} (1-\Phi_{2r}(\by_1,\by_2))\abs{ \dvol_g(\by_1,\by_2) } \\
\leq &  \int_{\Omega_1} h_r^{\epsilon}( \by_1) h_r^{\epsilon}( \by_2) e^{-(1- (d\nu' r + \nu r^2) ) \Lambda} \abs{ \dvol_g(\by_1,\by_2) } \\
 \leq & C e^{-(1- (d\nu' r + \nu r^2) ) \Lambda} \int_M \abs{ \dvol_g(c_1) } \\
 & \times \int_{r_1}^r d u_1  \int_{Gr(k,d)} u_1^{k(d-k)} d\mu_{k,d}(V_1) \int_{(\mathbb{S}_1(V_1))^{k+1}} u_1^{(k-1)(k+1)} \dvol_{(\mathbb{S}_1(V_1))^{k+1}}  \\ 
  & \times \int_0^{\alpha r} ds \int_{\mathbb{S}_1(T_{c_1}E)} s^{d-j} \dvol_{\mathbb{S}^{d-1}}(w)   \int_{r_1}^r d u_2  \int_{Gr(k-j,d)} u_i^{(k-j)(d-(k-j))} d\mu_{k-j,d}(W) \\
  & \times \int_{(\mathbb{S}_1(V_2))^{k+1-j}} u_2^{(k-1)(k+1-j)} \dvol_{(\mathbb{S}_1(V_2))^{k+1-j}}  \ h_r^{\epsilon}( \exp_{c_1}(u_1\bw_1)) h_r^{\epsilon}( \exp_{c_2}(u_2\bw_2)) ,
\end{split}
\]
with $c_2= \exp_{c_1}(sw)$ and the points $\by_1 = \exp_{c_1}(u_1\bw_1)$, $\by_2= \exp_{c_2}(u_2\bw_2)$ determined as in the previous discussion. As before, the constant $C$  accounts for the metric-dependent terms and the $\Upsilon$-terms representing the parallelogram-volumes.  Integrating in the radial coordinates and taking $r_1=(1-\xi^2/2c_g^2)r$ with $\xi=\Lambda^{-1}$ we conclude that
\begin{eqnarray}\nonumber
I_j^{(1)}& \leq & C e^{-(1- (d\nu' r + \nu r^2) ) \Lambda} \alpha^{d-j+1}  r^{d(2k+1-j)+j(k-j)+1} \left( 1- \left( \frac{r_1}{r} \right)^{dk} \right) \left( 1- \left( \frac{r_1}{r} \right)^{d(k-j)+j(k+1-j)} \right) .
\end{eqnarray}
for some new $C$ depending on $j,k,d$ and the metric $g$. Given that $\Lambda r \to 0$ as $n \to + \infty$, the exponential term can be estimated as $e^{-\Lambda + (d\nu' r + \nu r^2) \Lambda} \sim e^{-\Lambda}$, which gives
\begin{eqnarray}\nonumber
I_j^{(1)} & \leq &  C \alpha^{d-j+1} \Lambda^{-4} e^{-\Lambda}  r^{d(2k+1-j)+j(k-j)+1} .
\end{eqnarray}
We now turn to  evaluate  the integral over $\Omega_2$. Firstly, notice that $\Vol(\by_1, \by_2)$ is increasing with $\dist(c_1,c_2)$ and so attains its minimum value when $\dist(c_1,c_2)=\alpha r$ in this set. Secondly, notice that for the term $(1-\Phi_{2r}(\by_1,\by_2))$ to be nonzero we must have $\dist(c_1,c_2)\le 4r$. Inserting $\dist(c_1,c_2)= \alpha r$ in the Taylor expansion \eqref{eq:vol_taylor} and using  the same kind of change of variables yields, we have
\[
\begin{split}
I_j^{(2)} &:=\int_{\Omega_2}  h_r^{\epsilon} ( \by_1)  h_r^{\epsilon}( \by_2)  e^{-n \Vol(\by_1, \by_2)}(1-\Phi_{2r}(\by_1,\by_2)) \abs{ \dvol_g(\by_1,\by_2) } \\ 
& \leq   \int_{\Omega_2} h_r^{\epsilon}( \by_1) h_r^{\epsilon}( \by_2) e^{-(1- (d\nu' r + \nu r^2) ) (1+ \omega_d \alpha / \omega_{d-1} ) \Lambda} \abs{ \dvol_g(\by_1,\by_2) } \\ 
 \leq & C  e^{- (1+ \omega_d \alpha / \omega_{d-1} ) \Lambda} \int_M \abs{ \dvol_g(c_1) } \\
 & \times \int_{r_1}^r d u_1  \int_{Gr(k,d)} u_1^{k(d-k)} d\mu_{k,d}(V_1) \int_{(\mathbb{S}_1(V_1))^{k+1}} u_1^{(k-1)(k+1)} \dvol_{(\mathbb{S}_1(V_1))^{k+1}}  \\ 
  & \times \int_{\alpha r}^{4r} ds \int_{\mathbb{S}_1(T_{c_1}E)} s^{d-j} \dvol_{\mathbb{S}^{d-1}}(w)   \int_{r_1}^r d u_2  \int_{Gr(k-j,d)} u_i^{(k-j)(d-(k-j))} d\mu_{k-j,d}(W) \\
  & \times \int_{(\mathbb{S}_1(V_2))^{k+1-j}} u_2^{(k-1)(k+1-j)} \dvol_{(\mathbb{S}_1(V_2))^{k+1-j}}  \ h_r^{\epsilon}( \exp_{c_1}(u_1\bw_1)) h_r^{\epsilon}( \exp_{c_2}(u_2\bw_2)) ,
\end{split}
\]
In this case the integration in the radial coordinates yields
\begin{eqnarray}\nonumber
I_j^{(2)} = \int_{\Omega_2}  h_r^{\epsilon} ( \by_1) h_r^{\epsilon}( \by_2)  e^{-n \Vol(\by_1, \by_2)} \abs{ \dvol_g(\by_1,\by_2) } \leq  C \Lambda^{-4} e^{-\Lambda} e^{- \omega_d \alpha \Lambda / \omega_{d-1} }  r^{d(2k+1-j)+j(k-j)+1}.
\end{eqnarray}
Putting these all together into \eqref{eq:Integral_I_j} we have
$$\meanx{I_j} \leq C n \Lambda^{2k+1-j} r^{j(k-j)+1} \Lambda^{-4} e^{-\Lambda} \left( e^{- \omega_d \alpha \Lambda / \omega_{d-1} } + \alpha^{d-j+1} \right)  . $$
\end{proof}

%===============================================================================

\section{Proof of the main theorem}\label{sec:proof}

%===============================================================================

In this section we combine the results proved in Sections \ref{sec:upper}--\ref{sec:second_moment}  to prove the main result - Theorem \ref{thm:main}. During the proof, we shall use the term ``with high probability'' (w.h.p.~), meaning that the probability goes to $1$ as $n\to\infty$. 

%===============================================================================

\begin{proof}[Proof of Theorem \ref{thm:main}] 
We divide the proof into two parts, corresponding to the upper and lower thresholds for the phase transition.\\

\noindent{\bf Upper threshold:}\\
Suppose that $\Lambda = \log n + k\log\log n + w(n)$, and take $r_0$ to satisfy the conditions in Lemma \ref{lem:crit_pts}.
Using Lemma \ref{lem:crit_pts} we have that $\mean{C_k(r,r_0)} \to 0$ and $\mean{C_{k+1}(r,r_0)}\to 0$, which implies (using Markov's inequality) that $\prob{C_k(r,r_0)>0} \to 0$ and $\prob{C_{k+1}(r,r_0)>0}\to 0$.
Using similar arguments to the ones used in the proof of Proposition \ref{prop:betti_order}, we can conclude that $H_k(\cC_r) \cong H_k(\cC_{r_0})$ w.h.p.~since there are no critical points of index $k$ and $k+1$ in $(r,r_0]$.
In addition, from \cite{flatto_random_1977} we know that at radius $r_0$ the union of balls $B_{r_0}(\cP)$ covers $M$ w.h.p.~, and therefore by the Nerve Lemma \ref{lem:nerve} we have that $H_k(\cC_{r_0}) \cong H_k(M)$. 
Thus, we conclude that $H_k(\cC_r)\cong H_k(M)$.

\noindent{\bf Lower threshold:}\\
From Proposition \ref{prop:second_moment} we know that w.h.p.~$\beta^\eps_k(r) >\frac{1}{2}\mean{\beta_k^{\eps}(r)}$. If $\Lambda = \log n + (k-2)\log\log n -w(n)$, then from Lemma \ref{lem:theta_cycle_order} we have that $\mean{\beta_k^{\eps}(r)} \ge a_k e^{w(n)}\to \infty$.
In addition, from Lemma \ref{lem:Betas} we know that $\beta_k(r) \ge \beta_k^\eps(r)$. Therefore, we have that w.h.p.~$\beta_k(r) > \beta_k(M)$, which implies that $H_k(\cC_r)\not\cong H_k(M)$.
\end{proof}

%===============================================================================

\section{Discussion}

%===============================================================================

\label{sec:disc}
Homological connectivity is one of the fundamental properties of random simplicial complexes.
In this paper we showed that the phase transition discovered in \cite{bobrowski_vanishing_2015}, applies to any compact Riemannian manifold.
This is due to the fact that any Riemannian metric can be locally well approximated by an Euclidean one. Our results suggest that the phase transition for homological connectivity exhibited by the \cech complex should occur at a critical value of $\Lambda$ which is inside the interval
\[	
[\log n + (k-2)\log\log n, \log n + k\log \log n].
\]
We note that the first and second-order terms are identical to those in \cite{bobrowski_vanishing_2015} describing the flat torus, while the lower order term could be different and depend on the Riemannian metric. Clearly, our work here is not done, as we are yet to have found the exact threshold for homological connectivity (for either the torus or general Riemannian manifolds). This, however, is left as future work. In particular, we  propose  the following conjecture which will be addressed in future work.

%===============================================================================

\begin{con}
Let $w(n)\to\infty$, then
\[
\limninf \prob{H_k(\cC_r) \cong H_k(M)} = \begin{cases} 1 & \Lambda = \log n + (k-1)\log\log n +w(n) \vspace{10pt}\\
0 & \Lambda = \log n + (k-1)\log\log n -w(n). \end{cases}
\]\end{con}

%===============================================================================

Our intuition for this conjecture is the following. In random graph models (e.g. Erd\"os-R\'enyi, and geometric), the obstruction to connectivity are isolated vertices. More specifically, in \cite{erdos_evolution_1960} it is shown that the graph becomes connected exactly at the point where the last isolated vertex connects to another vertex. A similar observation appeared later in both the Linial-Meshulam model \cite{linial_homological_2006} as well as the random clique complexes \cite{kahle_sharp_2014}. In both models it was shown that the obstruction to homological connectivity are ``isolated'' or ``uncovered'' $k$-faces ($k$-simplexes that are not a face of any $(k+1)$-simplexes).
These results suggest that the same phenomenon should occur in the random \cech complex.
Condition 2 in \eqref{eq:crit_pt_cond} implies that every critical point of index $k$ introduces an uncovered $k$-face. Thus, a possible candidate for the homological connectivity threshold is the point where the last critical point of index $k$ appears. Lemma \ref{lem:crit_pts} strongly suggests that this threshold is $\Lambda = \log n + (k-1)\log\log n$.

A final remark - in this paper we focused on extending the homological connectivity result from \cite{bobrowski_vanishing_2015} to compact Riemannian manifolds. However, the methods we used in this paper could be used to translate any of the previous statements made for random \cech (and also Vietoris-Rips) complexes \cite{bobrowski_distance_2014,bobrowski_topology_2014,kahle_random_2011,kahle_limit_2013,yogeshwaran_random_2016} from the Euclidean setup to Riemannian manifolds. For example, we could provide formulae for the expected Betti numbers in the sparse regime ($\Lambda \to 0$), as well as prove a central limit theorem (CLT) in either the sparse or the thermodynamic ($\Lambda = \lambda \in (0,\infty)$) regime.
For the sake of keeping this paper at a reasonable length, we did not include these statements here.

%===============================================================================

\section*{Acknowledgements}

%===============================================================================

The authors are very grateful to Robert Adler, Robert Bryant, Mark Stern and Shmuel Weinberger for helpful conversations about this and related work.

%===============================================================================

\newpage
\section*{Appendix}
\appendix

%===============================================================================

\section{Palm Theory for Poisson Processes}

%===============================================================================

The following theorem will be very useful when computing expectations related to Poisson processes. 

%===============================================================================

\begin{thm}[Palm theory for Poisson processes]
\label{thm:prelim:palm}
Let $(X,\rho)$ be a metric space, $f:X\to\R$ be a probability density on $X$, and let $\cP_n$ be a Poisson process on $X$ with intensity $\lambda_n = n f$.
Let $h(\cY,\cX)$ be a measurable function defined for all finite subsets $\cY \subset \cX \subset X^d$  with $\abs{\cY} = k$. Then
\[
    \E\Big\{\sum_{ \cY \subset \cP_n}
    h(\cY,\cP_n)\Big\} = \frac{n^k}{k!} \mean{h(\cY',\cY' \cup \cP_n)}
\]
where $\cY'$ is a set of $k$ $iid$ points in $X$ with density $f$, independent of $\cP_n$.
\end{thm}

%===============================================================================

For a proof of Theorem \ref{thm:prelim:palm}, see for example \cite{penrose_random_2003}.
We shall also need the following corollary, which treats second moments:

%===============================================================================

\begin{cor}\label{cor:prelim:palm2}
With the notation above, assuming $\abs{\cY_1} = \abs{\cY_2} = k$,
\[
    \E\Big\{\sum_{ \substack {
                    \cY_1 ,\cY_2\subset \cP_n  \\
                    \abs{\cY_1 \cap \cY_2} = j }}
    h(\cY_1,\cP_n)h(\cY_2,\cP_n)\Big\} = {\frac{n^{2k-j}}{j!((k-j)!)^2}} \mean{h(\cY_1',\cY' \cup \cP_n)h(\cY_2',\cY' \cup \cP_n)}
\]
where $\cY' = \cY'_1 \cup \cY'_2$ is a set of $2k-j$ $iid$ points in $X$ with density $f$, independent of $\cP_n$, and $\abs{\cY_1'\cap\cY_2'} = j$.
\end{cor}

%===============================================================================

For a proof of this corollary, see for example \cite{bobrowski_distance_2014}.

%===============================================================================

\section{Proofs for Section \ref{sec:ineqs}}\label{appendix1}

%===============================================================================

This appendix contains the proof of Lemmas \ref{lem:approx_sphere_vol}, \ref{lem:RiemannianBallsComparison} and corollary \ref{cor:Volume}.

%===============================================================================

\begin{proof}[Proof of Lemma \ref{lem:approx_sphere_vol}]
Using normal coordinates $(x^1 , \ldots , x^d)$ in a neighborhood of the point $p \in M$ we have $\dvol_g = \sqrt{ \det(g_{ij})} \ \dvol_{g_E}= dr \wedge r^{d-1} \sqrt{ \det (g_{ij})} \dvol_{{\mathbb{S}^{d-1}}}$, where $\mathbb{S}^{d-1}$ denotes the unit $(d-1)$-sphere in the Euclidean metric $g_E = \delta_{ij} dx^i \otimes dx^j$, and $\dvol_{\mathbb{S}^{d-1}}$ the volume form of the induced by the round metric on $\mathbb{S}^{d-1}$. Thus, we have that $\dvol_{{S_r(p)}}= \sqrt{\det(g_{ij})}r^{d-1} \dvol_{{\mathbb{S}^{d-1}}}$, and using  \eqref{eq:RiemannianMeasure} we have
\[
\dvol_{{S_r(p)}}  =  r^{d-1} \left( 1- \frac{Ric_{ij}}{3} x^i x^j + O(\abs{ x }^3) \right) \dvol_{{\mathbb{S}^{d-1}}}.
\]
For any $\nu>0$ we can find $r_{ \nu}(p)>0$, such that for all $r \leq r_{ \nu}(p)$
\[
\abs{ \dvol_{{S_r(p)}}  } \leq  r^{d-1} \left( 1 + \frac{\abs{ Ric } + \nu}{3} r^2 \right) \abs{  \dvol_{{\mathbb{S}^{d-1}}} },
\]
and similarly for the lower bound. Moreover, as the metric $g$ is smooth (i.e.~$C^{\infty}$) we can take $r_{\nu}(p)$ to depend  smoothly on $p \in M$. Then, we take $r_{\nu}= \min_{p\in M} r_{\nu}(p)$, which is achieved and positive as $M$ is {compact and one can chose $r_{\nu}(p)$ to vary continuously with $p \in M$, see remark \ref{rem:Continuous_r_nu}} below.
\end{proof}

%===============================================================================

\begin{rem}\label{rem:Continuous_r_nu}
Around any point $p \in M$ we can pick geodesic normal coordinates $(x^1,\ldots , x^d)$, valid in $B_{r}(p)$, for $r < \text{inj}(p)$, the injectivity radius at $p$\footnote{The injectivity radius at $p$ is the supremum of the values of $r$ such that $\exp_p : B_r(0) \subset T_pM \rightarrow M$ is injective.}. Then, for $r < \text{inj}(p)$ we may write $\dvol_{{S_r(p)}}=  r^{d-1} f(x^1,\ldots , x^d) \dvol_{{\mathbb{S}^{d-1}}}$. Moreover, for small $r$ the Taylor expansion of $f$ is such that
\[
f(x^1, \ldots , x^d)  =   1- \frac{Ric_{ij}}{3} x^i x^j -\frac{1}{9} \nabla_k Ric_{ij} x^i x^j x^k + \ldots.
\]
Then for any $\nu>0$, we have
\[
f(x^1, \ldots , x^d) \leq  1 + \frac{\abs{ Ric_{p} } + \nu}{3} r^2  ,
\]
provided that
\begin{eqnarray*}
0 & \leq & \left( 1 + \frac{\abs{ Ric_{p} } + \nu}{3} r^2  \right) - \left( 1- \frac{Ric_{ij}}{3} x^i x^j -\frac{1}{9} \nabla_k Ric_{ij} x^i x^j x^k + \ldots \right) \\
& \leq & \frac{\abs{ Ric_{p} } r^2 - Ric_{ij}x^i x^j + \nu r^2}{3}  + c_3(p) r^3  \\
& \leq & \frac{\nu r^2}{3}  + (c_3(p)+1) r^3 
\end{eqnarray*}
where $c_{3}(p) = \sup_{\abs{ v } =1}  \left(  \frac{1}{9} (\nabla_v Ric)(v,v) \right)$. Solving this inequality for $r$ we find that it is enough that 
$$r \leq r_{\nu}(p) = \min \lbrace \frac{\nu}{3 (\abs{ c_3(p) } +1)} ,\text{inj}(p) \rbrace .$$
As $g$ is a smooth metric, both $c_3(\cdot)$ and $\text{inj}(\cdot)$ are continuous and so the minimum of the two is continuous. It will be important for the proof of the previous result that one can choose this $r_{\nu}(p)$ to vary continuously with $p$.
%If we further want $r_{\nu}(p)$ to be smooth we can simply take a smooth aproximation to this minimum from below.
\end{rem}

%===============================================================================

\begin{proof}[Proof of Corollary \ref{cor:Volume}]
For every point $p \in M$, we have  from \eqref{eq:RiemannianMeasure} that $\sqrt{ \det(g_{ij}) } = 1- \frac{Ric_{ij}}{3} x^i x^j + O(\abs{ x }^3)$. Thus, if we take $\nu(p) = \abs{ Ric (p) } + \delta$ for some fixed $\delta>0$, then we can find $r_{\nu}(p)>0$  such that Lemma \ref{lem:RiemannianMeasureComparison} holds. To complete the proof we take $\nu = \min_{p\in M} \nu(p)$, and $r_\nu = \max_{p\in M} r_\nu(p)$. 
\end{proof}

%===============================================================================

\begin{proof}[Proof of Lemma \ref{lem:RiemannianBallsComparison}]
%We prove the case of the union of two balls as the case of the intersection can easily be obtained from a verbatim of that one, switching union to intersection.
Let $p_1,p_2$ be at a sufficiently small distance from each other, and let $p$ be the midpoint in the minimizing geodesic connecting $p_1$ and $p_2$. For $\dist(p_1,p_2)<2s$ we have $B_{s}(p_1) \cup B_{s}(p_2) \subset B_{2s}(p)$. Pick normal coordinates $(x^1, \ldots , x^d)$ centered at $p$ and valid in a ball of radius $2s$ centered around it. In these coordinates, the metric $g$ is within $O(s^2)$ of the Euclidean metric and its distance function $\dist_g$ is within $O(s)$ of the Euclidean one $\dist_E$, i.e.~$\dist_g/ \dist_E = 1+O(s)$. Hence, there exists $\nu_p >0$ such that 
%if $d_i (\cdot)$ for $i=1,2$, then
$$(1- \nu_p s) \rho_{p_i}^E \leq \rho_{p_i} \leq (1+ \nu_p s) 
\rho_{p_i}^E,\quad i=1,2. $$
From this, it immediately follows that for any $r \leq s$
$$\left ( B^E_{(1- \nu_p s)r}(p_1) \cup B^E_{(1- \nu_p s)r}(p_2) \right) \subset \left( B_{r}(p_1) \cup B_{r}(p_2) \right) \subset \left( B^E_{(1+ \nu_p s)r}(p_1) \cup B^E_{C(1+ \nu_p s)r}(p_2) \right).$$
The result then follows from using the compactness of $M$, and  maximizing $\nu_p$ over all $p_1,p_2 \in M$ within distance $2s$ from each other.
\end{proof}

%===============================================================================

\section{A convexity result}\label{appendix2}

%===============================================================================

In this section we present a result regarding equidistant hypersurfaces, i.e.~set of points which are equidistant to a fixed pair points. In the Euclidean case this is simply a hyperplane and thus a totally geodesic submanifold. In the following we prove that in a Riemannian manifolds, if two points are sufficiently close, then their equidistant hypersurface is approximately totally geodesic.  This result is used in the proof of  Lemma \ref{lem:center}, where we restrict a strictly convex function to a small open set in an equidistant hypersurface. Our result here guarantees that the restricted function is locally convex and thus admits a unique minimum.

%===============================================================================

\begin{lem}\label{lem:Strictly_Convex}
Let $p \in M$, then there exists $r_p>0$ that satisfies the following. For $p_1,p_2 \in B_{r_p}(p)$ let
$$E_{p_1, p_2}= \lbrace x \in M \ \vert \ \dist(p_1,x)= \dist(p_2,x) \rbrace$$
be the equidistant hypersurface between $p_1$ and $p_2$. Then, $\dist^2(p, \cdot)$ restricted to $E_{12} \cap B_{r_p}(p)$ is strictly convex.
\end{lem}

%===============================================================================

\begin{proof}
Fix normal coordinates $ x= (x^1, \ldots , x^d)$ in a $\delta$-neighborhood of $p$ with $x(p)=0$ and let
$$s_{\delta} (x) = \delta x , $$
for $(x^1)^2+ \ldots (x^d)^2 \leq 1$.  Using the pullback $s_{\delta}^* g$ of the metric $g$ to the Euclidean unit ball,  we define the metric $g_{\delta}:= \delta^{-2} s_{\delta}^* g$. In the coordinates $(x^1, \ldots , x^{d})$ we can write $g_{\delta} = g_{ij}^{\delta} dx^i \otimes dx^j$ with
\begin{equation}\label{eq:Rescalled_Metric}
g_{ij}^{\delta} (x) =  g_{ij}(\delta x) = \delta_{ij} + \frac{\delta^2}{3} R_{iklj} x^k x^l + O(\delta^3 ).
\end{equation}
Moreover, given $p_1 \in \B_{\delta}(p)$, we can write the distance function with respect to $g_{\delta}$ in the coordinates $(x^1,\ldots , x^d)$ as
\begin{equation}\label{eq:Equidistant_Hypersurface}
{\delta} \dist_{\delta}(\delta^{-1} p_1, \delta^{-1} x) = \dist (p_1,x),
\end{equation}
where $ \dist (p_1,x)$ denotes the distance function of $g$ in the same coordinates. The reason for using  this rescaling is twofold:
\begin{itemize}
\item The metric $g_{\delta}$ converges uniformly to the Euclidean metric, in the sense that $g^{\delta}_{ij} \rightarrow \delta_{ij}$ as $\delta \rightarrow 0$. This is a direct consequence of equation \eqref{eq:Rescalled_Metric} Moreover,  this way the coordinates $x$ are fixed and their range is not shrinking.
\item Let $p_1,p_2 \in B_{\delta}(p)$. The equidistant hypersurfaces $E_{p_1,p_2}$ and $E^{\delta}_{\delta^{-1}p_1,\delta^{-1}p_2}$ of the metrics $g$ and $g_{\delta}$ inside $B_1(p)$ are related by $E^{\delta}_{\delta^{-1}p_1,\delta^{-1}p_2}= \delta^{-1} E_{p_1,p_2}$. This claim is immediate from equation \eqref{eq:Equidistant_Hypersurface}.
\end{itemize}
We wish prove that, restricted to $E_{p_1, p_2} \cap B_{\delta}(p)$, the function $\dist(p, \cdot)$ is convex. Recall that $p=0$ in this coordinates. Then, given the relations above and the fact that $\dist_{\delta}(p,\cdot)= \delta^{-1} \dist(p, \delta \cdot)$ it will be enough to show that $\dist_{\delta}(p, \cdot)$ restricted to $E_{\delta^{-1}p_1,\delta^{-1}p_2}^{\delta} \cap B_1(p)$ is strictly convex.

From the first bullet above it follows that $\dist_{\delta}(p, \cdot)$ varies smoothly with $\delta$ and agrees with the Euclidean distance function $\sqrt{(x^1)^2+ \ldots + (x^d)^2}$ when $\delta =0$. Hence, using a Taylor expansion in $\delta$ we have
$$\dist_{\delta}^2(p,x)= (x^1)^2 + \ldots (x^d)^2 + \delta f(x, \delta),$$
where $f(x,\delta)$ is smooth in both variables, and uniformly bounded in a neighborhood of $\delta=0$.
Similarly, the equidistant hypersurfaces vary smoothly with $\delta$ and at $\delta=0$ must coincide with the Euclidean equidistant hyperplane between $p_1$ and $p_2$. Using a Taylor expansion we have
$$\dist_{\delta}^2(p_1, x)-\dist_{\delta}^2(p_2,x)= H_{12}(x)+ \delta h(x, \delta),$$
where $H_{12}$ is an affine function and  $h(x, \delta)$ is a smooth function, uniformly bounded around $\delta=0$.  Restricting $\dist_{\delta}^2(p,x)$ to $E_{\delta^{-1}p_1,\delta^{-1}p_2}^{\delta} = E_{12}^{\delta}$ and using the induced metric on $E_{12}^{\delta}$, we can regard the Hessian $H_{\delta}$ of $\dist_\delta^2$ as an endomorphism of $TE_{12}^{\delta}$. For $\delta=0$ we have a restriction of a strictly convex function to a totally geodesic submanifold (an hyperplane), and therefore $H_0$ is positive definite. Thus, for small $\delta$ we have
$$H_{\delta}= H_0 + \delta H_1 + O(\delta^2),$$
and there exists $r_p>0$ such that for $\delta< r_p$ we have that $H_{\delta}$ is still positive definite, which implies  that the restriction of $\dist_{\delta}^2(p,x)$ to $E_{12}^{\delta}$ is strictly convex.
\end{proof}

%===============================================================================

\section{An excess inequality}

%===============================================================================

In this section we prove an excess inequality for certain sets of points $\mathcal{Y}$ in a sufficiently small normal ball. This inequality is used in the proof of Lemma \ref{lem:Betas}.
 Recall that using normal coordinates \eqref{eq:NormalCoords} centered at a point $p \in M$, the metric is approximately Euclidean to first order. Hence, $\rho^2_p$ approaches the squared Euclidean distance in a small enough neighborhood of $p$. 
 In the Euclidean space, the Hessian matrix satisfies $\nabla^2\rho_p^2 = 2I$, where $I$ is the identity matrix.
Therefore, for a Riemannian distance function, if $r>0$ is sufficiently small then there exists $A(r)>0$ such that on $B_{r}(p)$ we have
\begin{eqnarray}\label{eq:HesianInequality}
\nabla^2 \rho_p^2 \geq 2 A(r) I.
\end{eqnarray}
Moreover, the constant $A(r)$ converges to $1$ as $r \rightarrow 0$. We note that the constant $A$ can be taken to be greater than $1$ if $g$ has negative sectional curvatures at $p$, and it must be taken to be less than $1$ if one of the sectional curvatures is positive.

%===============================================================================

\begin{lem}\label{lem:intersect}
There exists a continuous function $c_g : (0, r_{\max}] \rightarrow \mathbb{R}$, that depends only the metric $g$, and satisfies the following. Let $\cY \subset \cP$ be such that $E_{r_{\max}}(\cY) \neq \emptyset$ (and $c(\cY)$ is well defined) and such  that $0\in \Delta(\cY)$. Then for every $r \in ( \rho(\cY), r_{\max} )$ and $x \in B_r^\cap(\cY)$, we have
\[
\dist(c(\cY), x) \leq { c_g(r) }\sqrt{r^2 - \rho^2(\cY)}.
\]
In addition, we have  that $\lim_{r \to 0} c_g(r)=1$.
\end{lem}

%===============================================================================

\begin{proof}
Let $\cY = \set{y_1,\ldots, y_k}$,   $x \in B_r^\cap(\cY)$, and  $\ell = \dist(c(\cY),x)$. In addition, let $\gamma: [0,\ell] \rightarrow M$ be the  minimizing geodesic from $c(\cY)$ to $x$, using the arc-length parametrization. For each $y_i \in \cY$ consider the squared distance from $y_i$ to a point $\gamma(t)$ on that geodesic (see Figure \ref{fig:geodesic}). Using the Hessian inequality \eqref{eq:HesianInequality}, and recalling that  $\gamma$ is parametrized with respect to arc-length, we have
\[
\frac{1}{2} \frac{d^2}{dt^2} (\rho_{y_i}^2 \circ \gamma)(t) = \frac{1}{2} ( \nabla^2_{\dot{\gamma}, \dot{\gamma}} \rho_{y_i}^2 )(\gamma(t)) \geq A(r) \abs{ \dot{\gamma}(t) }^2 = A(r),\quad 1\le i\le k,
\]
 Integrating this inequality from $0$ to $t$, using the fundamental theorem of calculus and multiplying by $2$ yields $\frac{d}{dt} (\rho_{y_i}^2 \circ \gamma)(t) \geq \langle \nabla \rho_{y_i}^2, \dot{\gamma}(0) \rangle + 2At$. Applying the fundamental theorem of calculus again we have
\[
\rho^2_{y_i}(\gamma(t))-\rho^2_{y_i}(\gamma(0))  \geq  \langle \nabla \rho_{y_i}^2, \dot{\gamma}(0) \rangle t + A(r)t^2 .
\]
Recall that  $\gamma(\ell)=x$, $\gamma(0)=c(\cY)$,  $\rho_{y_i}(x) \leq r$, $\rho_{y_i}(c(\cY))=\rho(\cY)$, and $\ell=\dist(c(\cY), x)$, then  putting everything into the last inequality we have
\begin{equation}\label{eq:excess_ineq}
A(r) \dist^2(c(\cY), x) \leq r^2 - \rho^2(\cY) -  \langle \nabla \rho_{y_i}^2, \dot{\gamma}(0)\rangle l,\quad 1\le i \le k.
\end{equation}
Finally, we use the assumption that $0\in \Delta(\cY)\subset T_{c(\cY)}M$, which implies that there exist $\lbrace \alpha_i \rbrace_{i=1}^{k}$ such that $\alpha_i \in [0,1]$, $\sum_i\alpha_i = 1$, and $\sum_{i=1}^{k} \alpha_i \nabla \rho_{y_i}^2 =0$. Multiplying each of the $k$ inequalities \eqref{eq:excess_ineq} by the corresponding $\alpha_i$ and summing them up  yields
\[
\dist(c(\cY), x)^2 \leq (r^2 - \rho^2(\cY))/A^2(r).
\]
To complete the proof, we set $c_g(r) = 1/\sqrt{A(r)}$. The continuity of $c_g(r)$ and the fact that is approaches $1$ as $r \to 0$  follow from the properties of $A(r)$.

\end{proof}

%===============================================================================

\begin{figure}[h]
\centering
\includegraphics[scale=0.4]{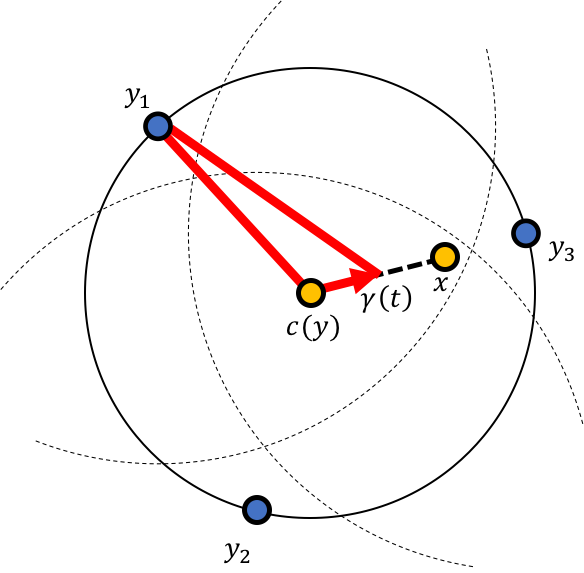}
\caption{\label{fig:geodesic}  The temporary construction of geodesics used in the proof of Lemma \ref{lem:intersect}. Here $\cY=\set{y_1,y_2,y_3}$, with $c(\cY)$ as its center. $\gamma(t)$ is a geodesic connecting $c(\cY)$ and $x\in B_r^\cap(\cY)$.}
\end{figure}

%===============================================================================

\bibliographystyle{plain}
\bibliography{refs}

\begin{thebibliography}{10}

\bibitem{adler_crackle:_2014}
Robert~J. Adler, Omer Bobrowski, and Shmuel Weinberger.
\newblock Crackle: {The} {Homology} of {Noise}.
\newblock {\em Discrete \& Computational Geometry}, 52(4):680--704, December
  2014.

\bibitem{balakrishnan_minimax_2012}
Sivaraman Balakrishnan, Alessandro Rinaldo, Don Sheehy, Aarti Singh, and
  Larry~A. Wasserman.
\newblock Minimax rates for homology inference.
\newblock In {\em {AISTATS}}, volume~9, pages 206--207, 2012.

\bibitem{bobrowski_distance_2014}
Omer Bobrowski and Robert~J. Adler.
\newblock Distance functions, critical points, and the topology of random {\v
  c}ech complexes.
\newblock {\em Homology, Homotopy and Applications}, 16(2):311--344, 2014.

\bibitem{bobrowski_maximally_2015}
Omer Bobrowski, Matthew Kahle, and Primoz Skraba.
\newblock Maximally {Persistent} {Cycles} in {Random} {Geometric} {Complexes}.
\newblock {\em To appear in: The Annals of Applied Probability.
  arXiv:1509.04347}, September 2015.
\newblock arXiv: 1509.04347.

\bibitem{bobrowski_topology_2014}
Omer Bobrowski and Sayan Mukherjee.
\newblock The topology of probability distributions on manifolds.
\newblock {\em Probability Theory and Related Fields}, 161(3-4):651--686, 2014.

\bibitem{bobrowski_topological_2017}
Omer Bobrowski, Sayan Mukherjee, and Jonathan~E. Taylor.
\newblock Topological consistency via kernel estimation.
\newblock {\em Bernoulli}, 23(1):288--328, February 2017.

\bibitem{bobrowski_vanishing_2015}
Omer Bobrowski and Shmuel Weinberger.
\newblock On the {Vanishing} of {Homology} in {Random} {\v c}ech {Complexes}.
\newblock {\em To appear in: Random Structures \& Algorithms. arXiv preprint
  arXiv:1507.06945}, 2015.

\bibitem{borsuk_imbedding_1948}
Karol Borsuk.
\newblock On the imbedding of systems of compacta in simplicial complexes.
\newblock {\em Fundamenta Mathematicae}, 35(1):217--234, 1948.

\bibitem{carlsson_topology_2009}
Gunnar Carlsson.
\newblock Topology and data.
\newblock {\em Bulletin of the American Mathematical Society}, 46(2):255--308,
  2009.

\bibitem{chazal_sampling_2009}
Fr{\'e}d{\'e}ric Chazal, David Cohen-Steiner, and Andr{\'e} Lieutier.
\newblock A sampling theory for compact sets in {Euclidean} space.
\newblock {\em Discrete \& Computational Geometry}, 41(3):461--479, 2009.

\bibitem{doCarmo}
MP~Do~Carmo.
\newblock {\em Riemannian Geometry. Mathematics: Theory \& Applications.
  Birkhuser Boston}.
\newblock 1992.

\bibitem{duy_limit_2016}
Trinh~Khanh Duy, Yasuaki Hiraoka, and Tomoyuki Shirai.
\newblock Limit theorems for persistence diagrams.
\newblock {\em arXiv:1612.08371 [math]}, December 2016.
\newblock arXiv: 1612.08371.

\bibitem{edelsbrunner2016expected}
Herbert Edelsbrunner, Anton Nikitenko, and Matthias Reitzner.
\newblock Expected sizes of poisson-delaunay mosaics and their discrete morse
  functions.
\newblock {\em arXiv preprint arXiv:1607.05915}, 2016.

\bibitem{erdos_random_1959}
Paul Erd{\H o}s and Alfr{\'e}d R{\'e}nyi.
\newblock On random graphs.
\newblock {\em Publicationes Mathematicae Debrecen}, 6:290--297, 1959.

\bibitem{erdos_evolution_1960}
Paul Erd{\H o}s and Alfr{\'e}d R{\'e}nyi.
\newblock On the evolution of random graphs.
\newblock {\em Magyar Tud. Akad. Mat. Kutat{\'o} Int. K{\"o}zl}, 5:17--61,
  1960.

\bibitem{flatto_random_1977}
Leopold Flatto and Donald~J. Newman.
\newblock Random coverings.
\newblock {\em Acta Mathematica}, 138(1):241--264, 1977.

\bibitem{gershkovich_morse_1997}
Vladimir Gershkovich and Haim Rubinstein.
\newblock Morse theory for {Min}-type functions.
\newblock {\em Asian Journal of Mathematics}, 1:696--715, 1997.

\bibitem{gilbert_random_1961}
Edward~N. Gilbert.
\newblock Random plane networks.
\newblock {\em Journal of the Society for Industrial \& Applied Mathematics},
  9(4):533--543, 1961.

\bibitem{Gromov1981}
Michael Gromov.
\newblock Curvature, diameter and betti numbers.
\newblock {\em Commentarii Mathematici Helvetici}, 56(1):179--195, 1981.

\bibitem{Grove1977}
Karsten Grove and Katsuhiro Shiohama.
\newblock A generalized sphere theorem.
\newblock {\em Annals of Mathematics}, 106(1):201--211, 1977.

\bibitem{hatcher_algebraic_2002}
Allen Hatcher.
\newblock {\em Algebraic topology}.
\newblock Cambridge University Press, 2002.

\bibitem{HopfRinow}
H.~Hopf and W.~Rinow.
\newblock Ueber den {B}egriff der vollst\"andigen differentialgeometrischen
  {F}l\"ache.
\newblock {\em Comment. Math. Helv.}, 3(1):209--225, 1931.

\bibitem{kahle_topology_2009}
Matthew Kahle.
\newblock Topology of random clique complexes.
\newblock {\em Discrete Mathematics}, 309(6):1658--1671, 2009.

\bibitem{kahle_random_2011}
Matthew Kahle.
\newblock Random geometric complexes.
\newblock {\em Discrete \& Computational Geometry}, 45(3):553--573, 2011.

\bibitem{kahle_sharp_2014}
Matthew Kahle.
\newblock Sharp vanishing thresholds for cohomology of random flag complexes.
\newblock {\em Annals of Mathematics}, 179(3):1085--1107, May 2014.

\bibitem{kahle_topology_2014}
Matthew Kahle.
\newblock Topology of random simplicial complexes: a survey.
\newblock {\em AMS Contemp. Math}, 620:201--222, 2014.

\bibitem{kahle_limit_2013}
Matthew Kahle and Elizabeth Meckes.
\newblock Limit the theorems for {Betti} numbers of random simplicial
  complexes.
\newblock {\em Homology, Homotopy and Applications}, 15(1):343--374, 2013.

\bibitem{Kobayashi}
Shoshichi Kobayashi and Katsumi Nomizu.
\newblock {\em Foundations of Differential geometry Volume 1}.
\newblock Wiley Classics Library, 1996.

\bibitem{Lang}
Serge Lang.
\newblock {\em Differential and Riemannian Manifolds}.
\newblock Springer, 1995.

\bibitem{linial_homological_2006}
Nathan Linial and Roy Meshulam.
\newblock Homological connectivity of random 2-complexes.
\newblock {\em Combinatorica}, 26(4):475--487, 2006.

\bibitem{meshulam_homological_2009}
Roy Meshulam and Nathan Wallach.
\newblock Homological connectivity of random k-dimensional complexes.
\newblock {\em Random Structures \& Algorithms}, 34(3):408--417, 2009.

\bibitem{miles1971isotropic}
Roger~Edmund Miles.
\newblock Isotropic random simplices.
\newblock {\em Advances in Applied Probability}, 3(02):353--382, 1971.

\bibitem{milnor_morse_1963}
John~Willard Milnor.
\newblock {\em Morse theory}.
\newblock Princeton university press, 1963.

\bibitem{munkres_elements_1984}
James~R. Munkres.
\newblock {\em Elements of algebraic topology}, volume~2.
\newblock Addison-Wesley Reading, 1984.

\bibitem{niyogi_finding_2008}
Partha Niyogi, Stephen Smale, and Shmuel Weinberger.
\newblock Finding the homology of submanifolds with high confidence from random
  samples.
\newblock {\em Discrete \& Computational Geometry}, 39(1-3):419--441, 2008.

\bibitem{niyogi_topological_2011}
Partha Niyogi, Stephen Smale, and Shmuel Weinberger.
\newblock A topological view of unsupervised learning from noisy data.
\newblock {\em SIAM Journal on Computing}, 40(3):646--663, 2011.

\bibitem{owada_limit_2015}
Takashi Owada and Robert~J. Adler.
\newblock Limit {Theorems} for {Point} {Processes} under {Geometric}
  {Constraints} (and {Topological} {Crackle}).
\newblock {\em To appear in: The Annals of Probability. arXiv preprint
  arXiv:1503.08416}, 2015.

\bibitem{penrose_random_2003}
Mathew Penrose.
\newblock {\em Random geometric graphs}, volume~5.
\newblock Oxford University Press Oxford, 2003.

\bibitem{Peterson}
Peter Peterson.
\newblock {\em Riemannian Geometry}, volume 171.
\newblock Springer, 2006.

\bibitem{Viaclovsky}
J~Viaclovsky.
\newblock Topics in riemannian geometry. notes of curse math 865, fall 2011,
  2011.

\bibitem{wasserman_topological_2016}
Larry Wasserman.
\newblock Topological {Data} {Analysis}.
\newblock {\em arXiv:1609.08227 [stat]}, September 2016.
\newblock arXiv: 1609.08227.

\bibitem{yogeshwaran_topology_2015}
D.~Yogeshwaran and Robert~J. Adler.
\newblock On the topology of random complexes built over stationary point
  processes.
\newblock {\em The Annals of Applied Probability}, 25(6):3338--3380, 2015.

\bibitem{yogeshwaran_random_2016}
D.~Yogeshwaran, Eliran Subag, and Robert~J. Adler.
\newblock Random geometric complexes in the thermodynamic regime.
\newblock {\em Probability Theory and Related Fields}, pages 1--36, 2016.

\bibitem{zomorodian_topological_2007}
Afra Zomorodian.
\newblock Topological data analysis.
\newblock {\em Advances in Applied and Computational Topology}, 70:1--39, 2007.

\end{thebibliography}

%===============================================================================

\end{document}